\documentclass{amsart}
\usepackage[utf8]{inputenc}
\usepackage[english]{babel}

\usepackage{mathtools}
\usepackage{amsmath}
\usepackage{amssymb}
\usepackage{amsthm}
\usepackage{mathabx} 
\usepackage{bbm}						
\usepackage{xcolor}		
\usepackage{graphicx} 
\usepackage{placeins}
\usepackage{appendix}
\usepackage{enumerate}
\usepackage{color}
\usepackage{todonotes}

\usepackage[top=3cm, bottom=3cm, inner=2.5cm, outer=2.5cm, marginparwidth=4cm]{geometry}

\usepackage{booktabs}															   																														

\usepackage{enumitem}

\usepackage{float}														\usepackage{wrapfig}	
\usepackage{upgreek}	
\usepackage[hidelinks]{hyperref}
\usepackage{subcaption}	
\theoremstyle{definition}

\newtheorem{lemma}{Lemma}[section]

\newtheorem{corollary}{Corollary}[section]
\newtheorem{claim}{Claim}[section]
\numberwithin{equation}{section}

\newtheorem{remark}{Remark}[section]
\usepackage{thmtools, thm-restate}
\usepackage{comment}

\begin{document}

\title[Unique continuation for Hyperbolic Schr\"odinger Equations]{Unique continuation and Hardy's Uncertainty Principle for Hyperbolic Schr\"odinger Equations}

\author[T. S. Jensen]{Torunn S. Jensen}

\address{Torunn S. Jensen, Department of Mathematics, University of Bergen, PO Box 7803, 5020 Bergen, Norway}
\email{Torunn.Jensen@uib.no}

\date{\today}

\maketitle
\section*{Abstract}
We prove unique continuation properties related to the Hardy uncertainty principle for solutions of the hyperbolic nonlinear Schrödinger equation and the hyperbolic Schrödinger equation with potential. Under suitable conditions on the nonlinearity, or the potential, we show that if $u$ is a solution with Gaussian decay at two different times, then 
$u\equiv 0$. These results extend to the hyperbolic setting the work of Escauriaza, Kenig, Ponce, and Vega (JEMS, 10, 2008) for the classical Schrödinger equation. The proofs rely on Carleman estimates based on calculus and convexity arguments, with the main challenge being to provide a rigorous justification of these estimates. Although our approach follows the general strategy of Escauriaza, Kenig, Ponce, and Vega, several technical modifications are required to handle the hyperbolic character of the equation. 

\section{Introduction}
\subsection{The hyperbolic Schr\"odinger equation} We consider the hyperbolic nonlinear Schrödinger equation (HNLS) 
\begin{equation}\label{HNLS}
    \partial_t u = i(Lu +F(u,\bar{u})),
\end{equation}
where $u=u(x,t) \in \mathbb C$, $x \in \mathbb R^n$, $t \in \mathbb R$, $F(u,\bar{u})$ is a complex-valued function, and, for $1 \le k < n$,
\begin{equation}\label{hyperbolic laplace}
    L=\sum_{j=1}^k \partial_{x_j}^2-\sum_{j=k+1}^n \partial_{x_j}^2,
\end{equation}
is the \lq\lq hyperbolic\rq\rq \, Laplace operator. These equations appear as asymptotic models in different physical contexts, such as water waves, plasma waves, electromagnetic waves and nonlinear optics. We refer to \cite{lannes_water_2013,saut_hyperbolic_2024,berge_wave_1998,kirane_scalar_2024, dumas_variants_2016}, and the references therein for general derivations of nonlinear hyperbolic Schrödinger-type equations. Of course, if $k=n$ in \eqref{hyperbolic laplace}, then $L=\Delta$, and we recover the classical nonlinear Schrödinger equation,
\begin{equation}\label{NLS}
    \partial_t u = i(\Delta u +F(u,\bar{u})).
\end{equation}

In the case where $F(u,\bar{u})=\pm|u|^2u$, we obtain the cubic hyperbolic NLS,
\begin{equation}\label{cubic HNLS}
\partial_t u = i(Lu + |u|^2u) .
\end{equation} 
For $n=2,$ the cubic hyperbolic NLS
$$\partial_t u = i(\partial_x^2-\partial_y^2 + |u|^2u)$$ describes the envelope of slowly modulated wave packets around highly oscillating waves in deep water. In nonlinear optics, it also models the evolution of electromagnetic fields in planar waveguides with self-focusing and normal dispersion, see \cite{gorza_ultrafast_2008,tan_wu_nonlinear_1993} and the references therein. The cubic hyperbolic NLS  also appears in more complicated settings, like the Davey-Stewartson system (see for example \cite{ghidaglia_initial_1990} for more details), where it is coupled with an elliptic/hyperbolic potential equation: 
\begin{align} \label{DS}
    \begin{cases}
        i\partial_t u + \delta \partial_x^2 u +\partial_y^2 u=\chi |u|^2u + b u \partial_x \phi, \\
        \partial^2_x \phi +m\partial^2_y \phi = \partial_x(|u|^2).
    \end{cases}
\end{align}
Here, the parameters $\delta, \chi, b, m$ are real, can be both positive and negative, and $|\delta|=|\chi|=1.$ \\

It was noticed in \cite{ghidaglia_nonelliptic_1993} that the Strichartz estimates for the linear part of \eqref{HNLS} are similar to the ones for the classical (elliptic) Schrödinger equation, so that the local well-posedness results for NLS (see for example \cite{linares_introduction_2015} and the references therein) also hold for \eqref{HNLS}. However, it is still a challenging open problem to understand the global behavior of the solutions to \eqref{HNLS}, see \cite{saut_hyperbolic_2024} for more details.\\ 

In this paper, we will derive unique continuation results for the hyperbolic nonlinear Schrödinger equation \eqref{HNLS}. In the process, we will first investigate the case of the hyperbolic Schrödinger equation with a potential 
\begin{equation} \label{HSEP}
        \partial_t u = i(Lu + V(x,t)u),
\end{equation}       
where the potential $V$ satisfies suitable assumptions (see Theorem \ref{hyperbolic result} below).

\subsection{Uncertainty principle for the Schrödinger equation }
In the series of works \cite{escauriaza_uniqueness_2006, escauriaza_convexity_2008, escauriaza_sharp_2010, escauriaza_morgan_2011, escauriaza_uniqueness_2011}, Escauriaza, Kenig, Ponce and Vega  proved several unique continuation properties for the Schrödinger equation with potential and the nonlinear Schrödinger equation.  These results generalize Hardy's Uncertainty principle from Fourier analysis: \textit{if $f(x) = O(e^{-|x|^2/\beta^2}), \  \hat{f}(\xi)=O(e^{-4|\xi|^2/\alpha^2})$ and $\alpha\beta < 4,$ then $f\equiv0$. Also, if $\alpha\beta=4$ then $f$ is a constant multiple of $e^{-|x|^2/\beta^2}$.} There is a corresponding $L^2$-result, proved in \cite{mauceri_generalisations_1983}: \textit{if $\|e^{|x|^2/\beta^2} f\|_{L^2(\mathbb{R})}$ and $\|e^{4|\xi|^2/\alpha^2} \hat{f} \|_{L^2(\mathbb{R})}$ are both finite and $\alpha \beta < 4$, then $f\equiv 0.$}
The extension of this result to $n$ dimensions has also been deduced using the Radon transform, see \cite{sitaram_uncertainty_1995}.

The solutions of the free Schrödinger equation write
$$u(x,t)=e^{it\Delta}u_0=\frac{e^{i|x|^2/4t}}{(2it)^{n/2}}\left(e^{i|\cdot|^2/4t}u_0\right)^{\wedge} \ (x/2t).$$
Thus $u(x,t)$ at any time $t$ is related to the Fourier transform of the initial data $u_0$. Then, we can apply the Hardy uncertainty principle to the function $f(x)=e^{i|x|^2/4}u_0$ and deduce the following unique continuation result:  \textit{If $u$ is a solution of the free Schrödinger equation, $\|e^{|x|^2/\beta^2}u_0\|_{L^2(\mathbb{R}^n)}$ and $\|e^{|x|^2/\alpha^2}u(1)\|_{L^2(\mathbb{R}^n)}$ are finite, and $\alpha\beta<4, $ then $u\equiv0$.}

For the hyperbolic case, the solution of the free equation writes

$$u(x,t) = e^{itL}u_0 = C(n,k,t)e^{i(|x_+|^2-|x_-|^2)/4t}*u_0,$$
for $|x_+|^2 = x_1^2+\dots+ x_k^2,$ and $|x_-|^2 = x_{k+1}^2+\dots + x_n^2,$ and for some constant $C$ depending on $n,k$ and $t.$
Thus, similar to the elliptic case,
$$u(x,t) = C(n,k,t)e^{i(|x_+|^2-|x_-|)/4t^2}\Big(e^{i/4t(|\cdot_+|^2-|\cdot_{-}|^2)}u_0\Big)^{\wedge}(\tilde{x}/2t),$$
where $\tilde{x}=(x_1,\dots,x_k,-x_{k+1},\dots,-x_n)=(x_+,-x_-).$ Hence, applying the Hardy uncertainty principle to the function $g(x)=e^{i (|x_+|^2-|x_-|^2)/4}u_0$ yields the Hardy uncertainty principle also for the solution of the free hyperbolic Schrödinger equation:
\textit{If $u$ is a solution of the free hyperbolic Schrödinger equation, $\|e^{|x|^2/\beta^2}u_0\|_{L^2(\mathbb{R}^n)}$ and $\|e^{|x|^2/\alpha^2}u(1)\|_{L^2(\mathbb{R}^n)}$ are finite, and $\alpha\beta<4, $ then $u\equiv0$.}\\

In \cite{escauriaza_hardys_2008}, the authors extended this result to the Schrödinger equation with potential and to the nonlinear Schrödinger equation. The proofs are quite different though and rely on Carleman's estimates. In particular, they proved the following: \\
\begin{restatable}[\cite{escauriaza_hardys_2008}]{theoremn}{EKPV}
\label{EKPV th1}
Let $u\in C([0,1],L^2(\mathbb{R}^n))$ be a solution to 
\begin{equation*}
    \partial_t u = i(\Delta u + V(x,t)u)  \ \ \ in \ \mathbb{R}^n \times [0,1],
\end{equation*}
where $V$ is bounded, and either $V(x,t)= V_1(x)+V_2(x,t)$ with $V_1$ real valued and
$$\sup_{t\in [0,1]}\|e^{\frac{|x|^2}{(\alpha t + \beta(1-t))^2}}V_2(t)\|_{L^\infty(\mathbb{R}^n)} < \infty, $$
or $$\lim_{R\to \infty}\int_{0}^1 \|V(t)\|_{L^\infty(\mathbb{R}^n\setminus B_{R})}dt=0.$$
If there exist constants $\alpha,\beta >0$ such that $\alpha \beta < 2$ and $\|e^{\frac{|x|^2}{\beta^2}}u(0)\|_{L^2(\mathbb{R}^n)} $ and   $\|e^{\frac{|x|^2}{\alpha^2}}u(1)\|_{L^2(\mathbb{R}^n)} $ are finite, then $u\equiv 0$. 
\end{restatable}

\begin{remark} This result is an extension of the Hardy Uncertainty Principle for the free Schrödinger equation to the Schrödinger equation with potential. The condition on the coefficients, $\alpha \beta <2$ was not sharp in \cite{escauriaza_hardys_2008}, and a bit weaker than the one from Hardy's uncertainty principle. However, in \cite{escauriaza_sharp_2010} the result was improved to be as sharp as in the free case, with the condition $\alpha\beta<4.$ 
\end{remark}

As a consequence, they also proved the following result for the NLS: 

\begin{restatable}[\cite{escauriaza_hardys_2008}]{theoremn}{EKPVNonlinear}
\label{Theorem 2 paper}
Let $u_1$ and $u_2$ be $C([0,1],H^k(\mathbb{R}^n))$ solutions to \begin{equation*}
    \partial_t u = i(\Delta u + F(u,\bar{u}))
\end{equation*} in $\mathbb{R}^n\times [0,1]$, with $k\in\mathbb{Z}^+, \ k>n/2, \ F:\mathbb{C}^2\to \mathbb{C}, \ F\in C^k$ and $F(0)=\partial_uF(0)=\partial_{\bar{u}} F(0)=0.$ If there are positive constants $\alpha$ and $\beta$ with $\alpha\beta<2$ such that $\|e^{\frac{|x|^2}{\beta^2}}(u_1(0)-u_2(0))\|_{L^2(\mathbb{R}^n)}$ and $\|e^{\frac{|x|^2}{\alpha^2}}(u_1(1)-u_2(1))\|_{L^2(\mathbb{R}^n)}$  are finite. Then $u_1\equiv u_2.$
\end{restatable}

\subsection{Statement of the result}
In this work, we extend the Hardy uncertainty principle to solutions to the hyperbolic Schrödinger equation with potential and to solutions to the hyperbolic NLS, by applying the same techniques as in \cite{escauriaza_hardys_2008}.
We state our first result for the hyperbolic Schr\"odinger equation with potential. 
\begin{restatable}{theoremn}{Hyperbolicmainresult}
\label{hyperbolic result}
Let $u\in C([0,1],L^2(\mathbb{R}^n))$ be a solution to the problem
\begin{align*}
    \partial_t u = i(L u + V(x,t)u)  \ \ \ in \ \mathbb{R}^n \times [0,1]
\end{align*}
$L$ defined as in \eqref{hyperbolic laplace},
where $V$ is bounded, and either
\begin{equation}\label{cond1OR}\sup_{t\in [0,1]}\|e^{\frac{|x|^2}{(\alpha t + \beta(1-t))^2}}V(t)\|_{L^\infty(\mathbb{R}^n)} < \infty, \end{equation}
or \begin{equation} \label{cond2OR}\lim_{R\to \infty}\int_{0}^1 \|V(t)\|_{L^\infty(\mathbb{R}^n\setminus B_{R})}dt=0.\end{equation}
If there exist constants $\alpha,\beta >0$ such that $\alpha \beta < 2$ and $\|e^{\frac{|x|^2}{\beta^2}}u_0\|_{L^2(\mathbb{R}^n)} $ and $\|e^{\frac{|x|^2}{\alpha^2}}u(1)\|_{L^2(\mathbb{R}^n)} $ are finite, then $u\equiv 0$.  
\end{restatable}

\begin{remark} The first condition \eqref{cond1OR} on the potential $V$ is not as general as the one in Theorem \ref{EKPV th1}, which also includes a real time independent potential $V_1(x)$. We chose to restrain ourself for including this kind of potential to avoid technical difficulties related to the semi-group theory of the parabolic regularization of the hyperbolic Schrödinger propagator. 
\end{remark}

As a consequence of Theorem \ref{hyperbolic result}, we deduce the following result for the hyperbolic nonlinear Schrödinger equation, which in particular applies to the cubic hyperbolic NLS.

\begin{restatable}{theoremn}{Hyperbolicnonlinear}
\label{HNLS result}
Let $u_1$ and $u_2$ be $C([0,1],H^k(\mathbb{R}^n))$ solutions to the equation 
\begin{equation*}
\partial_t u = i(Lu + F(u,\bar{u})), 
\end{equation*} 
with $k\in\mathbb{Z}^+, \ k>n/2, \ F:\mathbb{C}^2\to \mathbb{C}, \ F\in C^k$ and $F(0)=\partial_uF(0)=\partial_{\bar{u}} F(0)=0.$ If there are positive constants $\alpha$ and $\beta$ with $\alpha\beta<2$ such that $\|e^{\frac{|x|^2}{\beta^2}}(u_1(0)-u_2(0))\|_{L^2(\mathbb{R}^n)},$ and $\|e^{\frac{|x|^2}{\alpha^2}}(u_1(1)-u_2(1))\|_{L^2(\mathbb{R}^n)}$  are finite, then $u_1\equiv u_2.$
\end{restatable}

The proof of Theorem \ref{EKPV th1} in \cite{escauriaza_hardys_2008} relies on Carleman's estimates. The techniques use only calculus and convexity arguments. However, the main challenge is that these computations are formal and providing a rigorous justification is not trivial. Formally, the proof of Theorem \ref{hyperbolic result} is almost the same as for Theorem \ref{EKPV th1} in \cite{escauriaza_hardys_2008}, but when it comes to making the arguments rigorous, some differences and technical difficulties occur. In the next subsection, we outline the main steps of the proof and highlight the key differences to the proof of Theorem \ref{EKPV th1}.\\

It is worth noting that in \cite{barcelo_hardy_2013}, Theorem \ref{EKPV th1} was extended for solutions of the electromagnetic Schrödinger equation 
$$\partial_t u = i(\Delta_A u + Vu),$$ where $\Delta_A= (\nabla-iA)^2$, and $A(x):\mathbb{R}^n\to\mathbb{R}^n$ is a real valued vector potential. In \cite{cassano_fanelli_sharp_2015} the result was also extended to be sharp.
In \cite{BarceloCassanoFanelli2024}, Barcelo, Cassano and Fanelli proved different unique continuation properties for solutions of the electromagnetic hyperbolic Schrödinger equation, by only assuming decay at one time instead of two. The techniques in \cite{BarceloCassanoFanelli2024} also rely on Carleman estimates and the use of the Appell transformation for the hyperbolic equation.

\subsection{Outline of the proof of Theorem \ref{hyperbolic result}.}The proof is divided into several steps.

\noindent \textit{Step 1: the conformal/Appell transformation.} The first step is to reduce the problem to the case where the parameters $\alpha$ and $\beta$ are equal. This can be done with the conformal/Appell transformation\footnote{The transformation in the hyperbolic case is not exactly the same as in the elliptic case, see in particular Lemma \ref{conformalappelltransformation}.}. Instead of assuming that $\|e^{\frac{|x|^2}{\beta^2}}u_0\|_{L^2(\mathbb{R}^n)}$ and $\|e^{\frac{|x|^2}{\alpha^2}}u(1)\|_{L^2(\mathbb{R}^n)}$ are finite for $\alpha \beta < 2$, we can assume that $\|e^{\gamma |x|^2}u_0\|_{L^2(\mathbb{R}^n)}$ and $\|e^{\gamma |x|^2}u(1)\|_{L^2(\mathbb{R}^n)}$ are finite for some $\gamma > \frac{1}{2}.$ \\

\noindent \textit{Step 2: heuristic argument.} The heuristic argument is the same as for the elliptic case in \cite{escauriaza_hardys_2008}. Let $u$ be a solution of the equation (\ref{HSEP}). We define $f=e^{\phi}u$ for some weight function $\phi=\phi_R,$ to be chosen later, depending on a large parameter $R$. Moreover, we let $H(t)=\|f(t)\|_{L^2(\mathbb{R}^n)}^2.$ Then $f$ satisfies the equation $$\partial_t f = (\mathcal{S+A})f,$$ for a symmetric operator $\mathcal{S}$ and a skew-symmetric operator $\mathcal{A},$ both depending on the weight function $\phi.$ Ideally, we would like to prove a log-convexity inequality for the function $H,$ by computing $\frac{d^2}{dt^2}\log{H(t)}.$ In particular, we want to choose $\phi$ such that $$\frac{d^2}{dt^2}\log{H(t)}\geq -h(R,\gamma),$$ where $h$ is a non-negative function depending on $\gamma$ and $\phi_R$. This will imply that the function $e^{-\frac{h(R,\gamma)}{2} t(1-t)}H(t)$ is logarithmically convex. Then, it follows that, for some $0<\epsilon<1,$
$$\|u(1/2)\|_{L^2(B_{R\epsilon/4})}\leq H(0)^{1/2}H(1)^{1/2}e^{-\tilde{h}(R,\gamma)}, \text{ where } \tilde{h}(R,\gamma)\longrightarrow +\infty \text{ when } R\to \infty \text{ and }\gamma > \frac{1}{2}.$$ If we let $R\to \infty,$ and $\gamma>1/2,$ the left-hand side goes to $\|u(1/2)\|_{L^2(\mathbb{R}^n)},$ while the right-hand side goes to 0 since $H(0)^{1/2}=\|e^{\gamma|x|^2}u_0\|_{L^2(\mathbb{R}^n)}$ and $H(1)^{1/2}=\|e^{\gamma |x|^2}u(1)\|_{L^2(\mathbb{R}^n)}$ are finite. This implies that $u(1/2)=0,$ and thus $u\equiv 0.$ \\

However, to rigorously justify this argument, we need to know in advance that $\|e^{\phi}u(t)\|_{L^2(\mathbb{R}^n)}$ is finite for all time $t\in [0,1]$ and for a suitable weight $\phi$. This is not obvious in general, even though we know that the weighted norm of the solution is finite at two times $0$ and $1.$ In \cite{escauriaza_hardys_2008}, a large part of the paper is dedicated to proving this result. Then the authors are able to conclude the proof by using a slightly different convexity argument, still relying on Carleman estimates. In this work, we follow the same path as in \cite{escauriaza_hardys_2008}. \\ 

\noindent \textit{Step 3: rigorous justification of the weighted norms of the solution for all $0 \le t \le 1$.}
The goal is to prove a persistence property for a Gaussian weighted $L^2$-norm of the solution $u$ of \eqref{HSEP}, as well as an average in time estimate of the Gaussian weighted $L^2$-norm of $\nabla u$, by assuming that the Gaussian weighted norm of $u$ is finite at times $t=0$ and $t=1,$ and under two different assumptions on the potential $V$. This result is the counterpart of Theorem 3 and Theorem 5 in \cite{escauriaza_hardys_2008} for the hyperbolic Schrödinger equation.
\begin{restatable}{theoremn}{Hth3} \label{theorem 3}
Let $u \in C([0,1],L^2(\mathbb R^n))$ be a solution of the hyperbolic Schrödinger equation with potential (\ref{HSEP}) such that, for some $\gamma>0$, 
\begin{equation} \label{01A}\|e^{\gamma|x|^2}u_0\|_{L^2(\mathbb{R}^n)} \text{ and 
 }\|e^{\gamma |x|^2}u(1)\|_{L^2(\mathbb{R}^n)}<\infty.\end{equation} If $V$ is a bounded potential satisfying either \begin{equation}\label{cond1}\sup_{t \in [0,1]}\|e^{\gamma |x|^2}V(t)\|_{L^\infty(\mathbb{R}^n)}<\infty \end{equation} or
 \begin{equation}\label{cond2} \lim_{R\to \infty} \|V\|_{L^1([0,1],L^\infty(\mathbb{R}^n\setminus B_R))} = 0,\end{equation}
 then for all $t\in [0,1]$,
\begin{equation}\label{IG1}\|e^{\gamma |x|^2}u(t)\|_{L^2(\mathbb{R}^n)} \leq N(V) \left(\|e^{\gamma |x|^2}u_0\|_{L^2(\mathbb{R}^n)}+\|e^{\gamma |x|^2}u(1)\|_{L^2(\mathbb{R}^n)}+\sup_{t\in [0,1]} \|u(t)\|_{L^2(\mathbb{R}^n)} \right),\end{equation}
and
\begin{equation} \label{IG2}\|\sqrt{t(1-t)}e^{\gamma|x|^2}\nabla u\|_{L^2(\mathbb{R}^n \times [0,1])} \leq N(V) \left( \|e^{\gamma |x|^2}u_0\|_{L^2(\mathbb{R}^n)} + \|e^{\gamma |x|^2}u(1)\|_{L^2(\mathbb{R}^n)} + \sup_{t\in [0,1]}\|u(t)\|_{L^2(\mathbb{R}^n)}\right),
\end{equation}
where \begin{equation*}
  N(V)=  \begin{cases}
 C_1(V) e^{N(C_1(V) + C_1(V)^2)} & \text{if $V$ satisfies }\eqref{cond1}, \\ 
        N (\sup_{t\in [0,1]}\|V(t)\|_{L^\infty(\mathbb{R}^n)})^2 & \text{if $V$ satisfies }\eqref{cond2},
    \end{cases}
\end{equation*} 
$C_1(V) = \sup_{t\in[0,1]} \|e^{\gamma |x|^2}V(t)\|_{L^\infty(\mathbb{R}^n)}e^{2\sup_{t\in [0,1]}\|Im V(t)\|_{L^\infty(\mathbb{R}^n)}}$ and $N$ is a positive constant.
\end{restatable}

To prove Theorem \ref{theorem 3}, we follow the strategy in \cite{escauriaza_hardys_2008}. The main idea is to perform a parabolic regularization on (\ref{HSEP}) and prove similar estimates to \eqref{IG1} and \eqref{IG2} for solutions of the regularized equation 
\begin{equation}\label{PAR}
\partial_t v = A\Delta v + iB(L v + Vv + F),
\end{equation}
where $A>0$.
In fact, if $u$ is a solution to \eqref{HSEP} we work with $u_\epsilon$, a solution to the regularized equation
\begin{align}\label{P epsilon}
 \partial_t u_\epsilon = \epsilon \Delta u_\epsilon + i(Lu_\epsilon + F_\epsilon),  
\end{align}
where $\epsilon>0,$ $F_\epsilon =e^{\epsilon \Delta t}(Vu),$ so that $u_\epsilon (t) = e^{\epsilon \Delta t}u(t).$ 

We will first prove energy- and Carleman estimates for solutions of the parabolic equation \eqref{PAR}, which we want to apply to our solution $u_\epsilon.$ Then we will take the limit as $\epsilon \to 0$ to conclude the proof of Theorem \ref{theorem 3}. 

\noindent \textit{Step 4: energy estimate for the regularized equation \eqref{PAR}.} We prove an energy estimate for a specific weight function $\phi(x,t)=s(t)|x|^2,$ where $s(t)=\frac{\gamma A}{A + 8\gamma(A^2+B^2)t}$: 
\begin{align}\label{EE}   
   \|e^{s(t)|x|^2} &  \nonumber u(T)\|_{L^2(\mathbb{R}^n)}^2 +A\| \nabla(e^{s(t)|x|^2} u)\|_{L^2(\mathbb{R}^n \times [0,T])}^2  +2A\| 2s(t)|x|e^{s(t)|x|^2} u\|_{L^2( \mathbb{R}^n \times [0,T])}^2 \\ & \leq e^{M_V+|B|} \left(\|e^{\gamma |x|^2} u_0\|_{L^2(\mathbb{R}^n)} +  \|e^{s(t)|x|^2} F(t)\|_{L^2(\mathbb{R}^n \times [0,T])}^2\right),
    \end{align}
    for  $\gamma>0$, $T\in [0,1]$ and $M_V=\sup_{t\in [0,1]}\|B\,\text{Im}\ V(t)\|_{L^{\infty}(\mathbb{R}^n)}$.

\begin{remark}
For this result, we only need the weighted $L^2$-norm to be finite at one time, ($t=0$ here). However, the weight we propagate for $t\geq 0$ is smaller than $e^{\gamma  |x|^2}$ and decreases with time.
\end{remark}
\begin{remark}
The estimate \eqref{EE} differs from the corresponding estimate in \cite{escauriaza_hardys_2008}. Here we chose to add the two last terms on the left-hand side of \eqref{EE}. Controlling these terms will be useful to justify the Carleman estimate for the gradient of $u$ in estimate \eqref{C2} in Step 6.
\end{remark}

\noindent \textit{Step 5: Carleman estimate for the regularized equation \eqref{PAR}.}
Now, we use the energy estimate (\ref{EE}) to prove that for a solution $u$ of (\ref{PAR}) satisfying (\ref{01A}), we have, for any time $0\leq t \leq 1$,

\begin{align} \label{C}
   \|e^{\gamma |x|^2} u(t)\|_{L^2(\mathbb{R}^n)}\leq e^{N((B^2(M_1^2 +M_2^2) + |B|(M_1+M_2)) } \|e^{\gamma |x|^2}u_0\|_{L^2(\mathbb{R}^n)}^{1-t}\|e^{\gamma |x|^2}u(1)\|_{L^2(\mathbb{R}^n)}^t,
\end{align}
where $M_1=\sup_t\|V(t)\|_{L^\infty(\mathbb{R}^n)}$ and 
\begin{equation} \label{def:M2}
M_2 = \sup_{t\in [0,1]} \frac{\|e^{\gamma |x|^2} F(t)\|_{L^2(\mathbb{R}^n)}}{\|u(t)\|_{L^2(\mathbb{R}^n)}}.
\end{equation}
This is a classical Carleman argument, and the main problem is to rigorously justify that $\|e^{\gamma|x|^2}u(t)\|_{L^2(\mathbb{R}^n)}$ is finite for $0<t<1$. In \cite{escauriaza_hardys_2008}, for the elliptic Schrödinger equation, the authors achieve this by introducing the weight 
$$ \tilde{\phi}_a(x)= \begin{cases} |x|^2 & |x|<1 \\ \frac{ 2|x|^{2-a}-a}{2-a} & |x|\geq 1, 
\end{cases}
$$
for $0<a<1$
and define $\tilde{\phi}_{a,\rho} = \tilde{\phi}_a *\theta_\rho$, for a radial mollifier $\theta_\rho$, such that at infinity, $e^{\gamma \phi_{a,\rho}}u $ does not grow faster than $e^{s(t)|x|^2}u$, where $s(t)$ is the weight obtained in the energy estimate (\ref{EE}). However, in the hyperbolic case, some new cross terms between the parabolic regularization $A\Delta$ and the hyperbolic operator $iBL$ appear in the expression of the commutator $[\mathcal{S,A}]$ in the Carleman estimate. It is not clear how to deal with these terms. To avoid this issue, we chose to work with a different weight function $\phi_a$, where these specific cross terms vanish. In particular, we introduce the function, 

\begin{align*}
\psi_a(x_j) :=
 \begin{cases}
x_j^2  & if \ |x_j|<1 \\
\frac{2|x_j|^{2-a}-a}{2-a} & if \ |x_j|\geq 1,
 \end{cases}
\end{align*}

\begin{align*}
\phi_a(x) := \sum_{j=1}^{n}\psi_a(x_j),
\end{align*}
where $x\in\mathbb{R}^n$, $x=(x_1,\dots,x_j,\dots,x_n).$ 

For this weight function we have that $\|e^{\gamma\phi_{a,\rho}}u(t)\|_{L^2(\mathbb{R}^n)}\leq \|e^{s(t)|x|^2}u(t)\|_{L^2(\mathbb{R}^n)}<\infty$ by \eqref{EE}, and thus we can rigorously justify the Carleman argument in this case. Finally, we conclude the proof of (\ref{C}) by letting $a$ and $\rho$ to $0$. \\ 

\noindent \textit{Step 6: Carleman estimate for the gradient of the solutions to \eqref{PAR}.}
We also need a similar result for $e^{\gamma |x|^2 } \nabla u$. In particular, we show that
\begin{align}\label{C2}
    \nonumber \|\sqrt{t(1-t)}e^{\gamma|x|^2}\nabla & u\|_{L^2(\mathbb{R}^n\times [0,1])} +  \|\sqrt{t(1-t)}|x|e^{\gamma |x|^2}u\|_{L^2(\mathbb{R}^n\times [0,1])} \nonumber \\ &\leq N(A,B, \gamma,M_1) [\sup_{t\in [0,1]} \|e^{\gamma |x|^2}u (t)\|_{L^2(\mathbb{R}^n)} + \sup_{t\in [0,1]} \|e^{\gamma |x|^2} F\|_{L^2(\mathbb{R}^n)}],
\end{align}
where the constant $N$ remains finite when $A^2+B^2$ is bounded away from $0$.
In \cite{escauriaza_hardys_2008} the authors did not include a rigorous justification for this argument. However, thanks to \eqref{EE}, the the rigorous justification of (\ref{C2}) follows by arguing as above in the justification of (\ref{C}), relying on the same arguments as those used to prove the Carleman estimate (\ref{C}).\\

\noindent \textit{Step 7: Applying the Carleman estimates to the solutions $u_\epsilon$ of \eqref{P epsilon}.} To apply estimate \eqref{C} to the solution $u_\epsilon$ to \eqref{P epsilon}, we need to verify that $M_2$, defined in \eqref{def:M2}, is finite. 
At this point, we need to distinguish between the two assumptions on the potential $V$. When $V$ satisfies \eqref{cond1}, then  
\[M_2 \le \sup_{t \in [0,1]}\frac{\|e^{\gamma |x|^2}V(t)\|_{L^\infty(\mathbb{R}^n)}\|u(t)\|_{L^2(\mathbb{R}^n)}}{\|u_{\epsilon}(t)\|_{L^2(\mathbb{R}^n)}} <+\infty .\] 
Thus, we can apply \eqref{C} to $u_\epsilon,$ and let $\epsilon\to 0$ to obtain \eqref{IG1} in the case of \eqref{cond1}. 

When the potential satisfies \eqref{cond2}, there is no easy way to verify that $M_2$ is finite. Thus, we will prove that $\|e^{\gamma |x|^2}u(t)\|_{L^2(\mathbb{R}^n)}$ is bounded using a different argument. For this aim, we follow the strategy of \cite{kenig_unique_2003}. We start by proving the following result:
there exists $\epsilon_0>0$ such that if $u$ is a solution to the equation
\begin{equation}
i\partial_t u + Lu = Vu + F
\end{equation} and $\|e^{\beta x_j}u_0\|_{L^2(\mathbb{R}^n)}$ and $\|e^{\beta x_j}u(1)\|_{L^2(\mathbb{R}^n)}$ are finite for some $\beta\in \mathbb{R}$, $1\le j \le n$, $$\|V\|_{L^1_tL^\infty_x}\leq \epsilon_0$$ and $$F\in L^1([0,1],L^2(\mathbb{R}^n))\cap L^1([0,1]),L^2(e^{2\beta x_j}dx)),$$  then for some $N>0$ independent of $\beta,$
\begin{equation}\label{beta}   \sup_{t\in [0,1]}\|e^{\beta x_j}u(t)\|_{L^2(\mathbb{R}^n)}\leq N\left( \|e^{\beta x_j}u_0\|_{L^2(\mathbb{R}^n)} + \|e^{\beta x_j}u(1)\|_{L^2(\mathbb{R}^n)}+\|F\|_{L^1([0,1],L^2(e^{2\beta x_j}dx)}\right).
\end{equation}
For $1\leq j\leq k$, the proof of \eqref{beta} follows exactly as in the elliptic case in \cite{kenig_unique_2003}, 
so we will focus on the case $k<j\leq n$, highlighting the main technical differences with the elliptic case.

Using \eqref{beta}, we obtain an exponential decay for the solution in arbitrary directions:
\begin{equation}\label{lambda}   \sup_{t\in [0,1]}\|e^{\lambda\cdot x}u(t)\|_{L^2(\mathbb{R}^n)}\leq N\left( \|e^{\lambda\cdot x}u_0\|_{L^2(\mathbb{R}^n)} + \|e^{\lambda\cdot x}u(1)\|_{L^2(\mathbb{R}^n)}+\|F\|_{L^1([0,1],L^2(e^{2\lambda\cdot x}dx)}\right)
\end{equation}
for almost every $\lambda\in \mathbb{R}^n$. In the elliptic case, the decay estimate \eqref{lambda} follows directly from the decay estimate \eqref{beta} in the $j=1$ direction and the invariance of the Laplacian under the orthogonal group $O(n)$.  In the hyperbolic case, $L$ is invariant under the indefinite/pseudo orthogonal group $O(k,n-k)$, (see Subsection \ref{pseudo group}). Therefore, estimate \eqref{beta} is needed both for $1 \le j \le k$ and $k<j \le n$. We refer to the proof of Lemma \ref{general KPV lemma} for more details.

By integrating \eqref{lambda} in $\lambda$ and splitting the potential $V$ into a small potential and a potential localized in a ball, we obtain a Gaussian weighted decay estimate for the solutions of \eqref{HSEP} with \eqref{cond2}, which proves \eqref{IG1}: 
\begin{equation}\label{gaussian}
    \sup_{t\in [0,1]}\|e^{\gamma |x|^2}u(t)\|_{L^2(\mathbb{R}^n)}\leq N \sup_{t\in [0,1]} \|V(t)\|_{L^\infty(\mathbb{R}^n)}\left(\|e^{\gamma |x|^2}u_0\|_{L^2(\mathbb{R}^n)}+\|e^{\gamma |x|^2}u(1)\|_{L^2(\mathbb{R}^n)} + \sup_{t\in [0,1]}\|u(t)\|_{L^2(\mathbb{R}^n)}\right) .
\end{equation}

To prove \eqref{IG2}, we then apply \eqref{C2} to $u_\epsilon$ with $\gamma_\epsilon=\gamma/(1+8\gamma\epsilon)$ and observe that
\begin{equation*}
\sup_{t \in [0,1]} \|e^{\gamma_\epsilon |x|^2} F_\epsilon\|_{L^2(\mathbb{R}^n)} \leq C \sup_{t \in [0,1]} \|e^{\gamma |x|^2} (Vu)\|_{L^2(\mathbb{R}^n)} <\infty
\end{equation*}
in both cases of the potential $V$, (where we used \eqref{gaussian} in the second case). Finally, by letting $\epsilon\to 0$, we obtain the result for $u$. \\

\noindent \textit{Step 7: proof of the main result, Theorem \ref{hyperbolic result}}. In this step, no new bad terms appear, since there is no parabolic regularization, and thus the argument proceeds in the same way as in \cite{escauriaza_hardys_2008}. We start by proving a Carleman estimate for compactly supported functions in both space and time. In particular, for $$\phi = \mu |x+Rt(1-t)\tilde{\xi}|^2+(1+\epsilon)R^2t(1-t)/16 \mu, \epsilon>0, \mu >0, R>0, \ |\xi|=1, \text{ and } g\in C^{\infty}_0(\mathbb{R}^{n+1}),$$ $$R\sqrt{\frac{\epsilon}{8\mu}} \|e^{\phi}g\|_{L^2(\mathbb{R}^{n+1})} \leq \|e^{\phi}(\partial_t-iL)g\|_{L^2(\mathbb{R}^{n+1})}. $$ 

Then, we introduce the function $g(x,t)=\theta_M(x) \eta_R(t) u(x,t),$ where $\theta_M$ and $\eta_R$ are compactly supported cutoff functions in space and time respectively, and satisfy $g=u$ in an open ball in $\mathbb{R}^{n+1},$ depending on a large parameter $R$ and some $\epsilon>0$, depending on $\gamma.$ By applying the Carleman estimate, it follows that
  \begin{align*} \label{Est. after using lemma7}
    R\|e^\phi g\|_{L^2(\mathbb{R}^n \times [0,1])} & \leq N(\epsilon,\mu,\gamma) \left(R  \sup_{t\in [0,1]} \|e^{\gamma |x|^2} \tilde{u} \|_{L^2(\mathbb{R}^n)}  + \frac{1}{M} e^{\gamma R^2/\epsilon} \|e^{\gamma |x|^2} (|u| + |\nabla u|)\|_{L^2(\mathbb{R}^n \times [\frac{1}{2R}, 1-\frac{1}{2R}])}\right).
\end{align*}
Note that the quantity $\|e^{\gamma|x|^2}(|u| + |\nabla u|)\|_{L^2(\mathbb{R}^n \times [\frac{1}{2R},1-\frac{1}{2R}])}$ is finite thanks to Theorem \ref{theorem 3}, so by letting $M\to +\infty,$ the last term on the right-hand side goes to 0. Inside the ball $B_{R, \epsilon}$ we can bound $\phi$ from below, such that 
$$\phi(x,t) \geq N(\epsilon,\gamma) R^2 > 0,$$
which will imply that for some constant $N(\epsilon, \gamma)$
\begin{equation*}
\|u\|_{L^2(B_{R, \epsilon} )}\leq e^{-N(\epsilon, \gamma) R^2}\sup_{t\in [0,1]} \|e^{\gamma |x|^2} u \|_{L^2(\mathbb{R}^n)}.
\end{equation*}
By this, we are able to conclude the proof by letting $R\to+\infty.$\\

The paper is structured as follows: we start Section 2 with notation and preliminaries. In Section 3, we prove Theorem \ref{theorem 3}, which corresponds to Steps 3-6. Finally, in Section 4, we prove Theorems \ref{hyperbolic result} and \ref{HNLS result} for the hyperbolic nonlinear Schrödinger equation.

\section{Notation and Preliminaries}

\subsection{General notation}
\begin{itemize}
    \item $N$ or $C$ will denote arbitrary positive constants, which can change from line to line. Sometimes we write $N_\gamma$, $C(\gamma, \epsilon)$, etc. for some parameters $\gamma, \epsilon$, to specify that the constants may depend on the specific parameters. If the constant matters, we will define it properly.
\item We sometimes use $a\lesssim b$ if $a\leq Cb$ for a constant not depending on any specific parameters.
    \item $ L^p, \ 1\leq p \leq \infty,$ denotes the usual Lebesgue space with norm $\|f\|_{L^p}$.

\item $H^s = W^{s,2}$ denotes the usual $L^2$-based Sobolev spaces.

\item The Fourier transform is defined by $$\hat{f}(\xi)=\frac{1}{(2\pi)^{n/2}}\int_{\mathbb{R}^n} e^{-i\xi \cdot x}f(x)dx.$$
\end{itemize}
\subsection{The \lq \lq hyperbolic\rq\rq   Laplacian}

We fix $k\in \mathbb{N}$, $1\leq k < n$  and define the \lq\lq hyperbolic \rq\rq \, Laplace operator \begin{equation*}L=\sum_{j=1}^{k}\frac{\partial^2}{\partial{x_j}^2} - \sum_{j=k+1}^{n}\frac{\partial^2}{\partial{x_j}^2}= \Delta_+ - \Delta_{-}.\end{equation*}
Moreover, we denote \begin{align*}
\nabla_+ := \left(\frac{\partial}{\partial x_1},\cdot \cdot \cdot, \frac{\partial}{\partial_{x_k}}\right),  \quad \nabla_{-} := \left(\frac{\partial}{\partial x_{k+1}},\cdot \cdot \cdot, \frac{\partial}{\partial_{x_n}}\right) 
\end{align*}
and
$$\nabla_H := \left(\nabla_+, -\nabla_{-}\right).$$
Then it follows that $$L=\nabla \cdot \nabla_H.$$ 
Similarly, for $x\in \mathbb{R}^n$ we define
$$x_+ :=(x_1, \dots, x_k), \quad 
x_{-}:=(x_{k+1},\dots, x_n)$$ and
$$\tilde{x}=(x_+,-x_-)=(x_1,\dots, x_k, -x_{k+1}, \dots, -x_n),$$
so that $|\tilde x|^2 = |x|^2$.
It follows from the definitions that
\begin{align*}
\nabla f \cdot \nabla_H g  = \nabla_H f \cdot \nabla g, \quad
|\nabla f|^2  = |\nabla_Hf|^2, \quad
L(fg) = Lf + Lg + 2\nabla_Hf\cdot \nabla g.
\end{align*}

We also define the matrix of second order derivatives
\begin{align*}
D^2_H f: = \begin{bmatrix} 
\partial^2_{x_1} f & \dots  &  \partial^2_{x_1,x_k} f  & - \partial^2_{x_1,x_{k+1}} f & \dots & -\partial^2_{x_1, x_n} f \\
    \vdots & \ddots & \vdots & \vdots & \ddots & \vdots\\
   \partial^2_{x_k, x_1} f &\dots  & \partial^2_{x_k} f & -\partial^2_{x_k, x_{k+1}}f & \dots & -\partial^2_{x_k, x_n} f \\ 
   -\partial^2_{x_{k+1},x_1} f& \dots & -\partial_{x_{k+1},x_k} f& \partial^2_{x_{k+1}} f & \dots &\partial^2_{x_{k+1},x_n} f \\ 
   \vdots & \ddots & \vdots & \vdots &\ddots & \vdots \\
   -\partial^2_{x_n, x_1} f & \dots & -\partial^2_{x_n, x_{k}} f & \partial^2_{x_n, x_{k+1}} f & \dots & \partial^2_{x_n} f
    \end{bmatrix}
\end{align*}
which looks like the Hessian matrix, but for $ 1\leq i \leq k$ and $k+1\leq j\leq n$ the mixed derivative $\partial^2_{x_i, x_j}$ takes a negative sign. If both $1\leq i, j \leq k$ or $k+1 \leq i,j \leq n$ then $\partial^2_{x_i, x_j}$ takes a positive sign. 

Observe that if f is a function for which $\partial^2_{x_i,x_j} f=0$ for all $i\neq j$, then $D^2_Hf = D^2f$ where $D^2f$ is the Hessian matrix of $f.$ \\

\subsection{The orthogonal and the indefinite orthogonal group} \label{pseudo group}

The orthogonal group $O(n)$ consists of $n\times n$ invertible matrices $A$ satisfying $A^TA=I$, where $I$ is the identity matrix. 

Let $J$ be the non-degenerate bilinear form represented by the matrix 
\begin{equation*}
J = \text{diag}(\underbrace{1,\dots,1}_{k},\underbrace{-1,\dots,-1}_{n-k}).
\end{equation*}
We define the indefinite orthogonal group $O(k,n-k)$, sometimes also called pseudo orthogonal group or generalized orthogonal group, to be the group of invertible $n\times n$ matrices $A$ such that
\begin{align}
    A^TJA=J.
\end{align}
In particular, the group $O(1,3)$ is known as the Lorentz group and is of interest in relativity. For more information, see, for example, Chapter 1 in \cite{hall_lie_2015}, Chapter 4 in \cite{oneill_semi-riemannian_1983} and the references therein.

For $\lambda \in \mathbb{R}^n$ we use the notation
\begin{equation}
     \langle \lambda,\lambda\rangle_{k,n-k} = \lambda^TJ\lambda = \lambda_1^2 +\dots +\lambda_k^2 -\lambda_{k+1}^2-\dots - \lambda_n^2.
\end{equation}
Observe that $\langle \lambda,\lambda\rangle_{k,n-k}$ can be positive, negative and $0$ depending on $\lambda$.

\begin{remark} Using this notation, we can also write
\begin{align*}
    \nabla_H &= J\nabla, \quad L=\nabla \cdot J\nabla, \quad  D^2_H = JD^2 J.
\end{align*}
\end{remark}
We now connect these groups with the
differential operators $\Delta$ and $L$. Since $\Delta$ is invariant under the orthogonal group $O(n)$, it follows that if $u$ is a solution to the equation
$$\partial_t u = i(\Delta u + V(x,t)u + F(x,t)),$$ then the function $v(x,t) = u(Ax,t)$, for $A\in O(n)$, satisfies the Schrödinger equation
$$\partial_t v = i (\Delta v +V(Ax,t)v + F(Ax,t)).$$
Similarly, the operator $L$ is invariant under the group $O(k,n-k)$. Indeed, if $A\in O(k,n-k)$, and $y=Ax$
then, we have by the chain rule

\begin{align*}
Lu(y) =\nabla_y^TJ\nabla_y u(y) =\nabla_y^T AJA^T\nabla _y u(y) =(A^T\nabla_y)^T J (A^T \nabla_y) u(y) = \nabla_x^T J\nabla_x (u(Ax) )=L(u\circ A)(x).
\end{align*}

Recall that the action of $O(n)$ on the Euclidean sphere $S^{n-1}_c =\{x\in \mathbb{R}^n: |x|^2=c\}$
$$O(n) \times S^{n-1}_c \to S^{n-1}_c $$
is transitive. This means that for every $\lambda\in \mathbb{R}^n\setminus \{0\}$, there is $A\in O(n)$ such that $A\lambda = \sqrt{c}e_1,$ for $c=|\lambda|^2.$ 

Consider now the hyperbolic space \begin{equation*}
        H_c^{n-1}(k):=\{x\in \mathbb{R}^n :  x_1^2+\dots+x_k^2-x_{k+1}^2 -\dots -x_n^2=\langle x,x\rangle_{k,n-k}=c\}.  
    \end{equation*}
If $c\neq 0,$ the action $$O(k,n-k)\times H^{n-1}_c(k)\to H^{n-1}_c(k)$$ is transitive (see for example Chapter 3 in \cite{lee_riemannian_1997}). In particular,  let $\lambda \in \mathbb{R}^n$, with $c=\langle\lambda,\lambda\rangle_{k,n-k}\neq 0$. Then $\lambda\in H_c^{n-1}(k)$, and there exist $A\in O(k,n-k)$ such that
$$
A\lambda=
\begin{cases}
    \pm\sqrt{c}e_1& \text{if } c>0,\\
    \pm \sqrt{-c}e_n &\text{if }c<0.
\end{cases}$$
\begin{remark} \ \smallskip
    \begin{itemize}
        \item The $\pm$ depends on the sign on $\lambda_1$ in the case $c>0,$ and the sign of $\lambda_n$ if $c<0.$
        \item $e_1$ can be replaced by $e_j$ for any $1\leq j\leq k$ and $e_n$ can be replaced by any $e_j$ for $k+1\leq j\leq n.$ 
        \item In the case $c=0$ the above argument does not apply.
        \end{itemize}
\end{remark}

\subsection{Some basic lemmas}
First we state some basic bounds in $L^2$ obtained from standard energy estimate, similar to the elliptic case.
\begin{lemma}\label{Energyestimate}
 Suppose $u\in C([0,1],L^2(\mathbb{R}^n))$ satisfies
    \begin{equation*}
   \begin{cases} \partial_t u = i(L u + V(x,t) u) \ \ \ \text{in} \ \mathbb{R}^n \times [0,1] \\ 
   u(x,0) = u_0,
   \end{cases}
    \end{equation*}
    then for $N_V=e^{\sup_{t\in [0,1]} \|Im V(t)\|_{L^\infty(\mathbb{R}^n)}}$, 
    \begin{equation}
        N_V^{-1}\|u_0\|_{L^2(\mathbb{R}^n)} \leq \|u(t)\|_{L^2(\mathbb{R}^n)} \leq N_V\|u_0\|_{L^2(\mathbb{R}^n)}.
    \end{equation}
\end{lemma}

Next, we restate the conformal/Appell transform for the hyperbolic Schrödinger equation from \cite{BarceloCassanoFanelli2024}.
\begin{lemma} \label{conformalappelltransformation} 
Let $$\partial_s u = i(L u + V(y,s)u + F(y,s)) \ \ \ \ \ \ \ \ \text{in} \ \mathbb{R}^n \times [0,1].$$
For $\alpha, \beta >0,$ $\gamma \in \mathbb{R}$, define $$\tilde{u}(x,t) =\Big(\frac{\sqrt{\alpha \beta}}{\alpha(1-t)+\beta t}\Big)^{n/2} u\Big(\frac{\sqrt{\alpha \beta } x}{\alpha(1-t)+\beta t}, \frac{\beta t} {\alpha(1-t)+\beta t}\Big) e^{\frac{(\alpha-\beta)(|x_+|^2-|x_{-}|^2)}{4i(\alpha(1-t)+\beta t)}}.$$
Then $\tilde{u}$ satisfies 
$$\partial_t \tilde{u} = i(L\tilde{u} + \tilde{V}(x,t)\tilde{u} + \tilde{F}(x,t)) \ \ \ \ \ \ \ \text{in } \mathbb{R}^n \times [0,1],$$
where
$$\tilde{V}(x,t) = \frac{\alpha \beta}{(\alpha(1-t)+\beta t)^2}V\Big(\frac{\sqrt{\alpha \beta } x}{\alpha(1-t)+\beta t}, \frac{\beta t} {\alpha(1-t)+\beta t}\Big)$$
$$\tilde{F}(x,t) = \left(\frac{\sqrt{\alpha \beta}}{\alpha(1-t)+\beta t}\right)^{n/2 +2}F\Big(\frac{\sqrt{\alpha \beta } x}{\alpha(1-t)+\beta t}, \frac{\beta t} {\alpha(1-t)+\beta t}\Big) e^{\frac{(\alpha-\beta)(|x_+|^2-|x_{-}|^2)}{4i(\alpha(1-t)+\beta t)}}$$
Moreover, 
\begin{equation} \label{weighted est. for F A transform}\|e^{\gamma|x|^2}\tilde{F}(t)\|_{L^2(\mathbb{R}^n)}=\frac{\alpha \beta}{(\alpha(1-t)+\beta t)^2} \|e^{\frac{\gamma \alpha \beta}{(\alpha s + \beta(1-s))^2} |x|^2} F(s)\|_{L^2(\mathbb{R}^n)},
\end{equation}
and
\begin{equation}\label{weighted est. for u A transform}
\|e^{\gamma |x|^2} \tilde{u}(t)\|_{L^2(\mathbb{R}^n)} = \|e^{\frac{\gamma \alpha \beta}{(\alpha s + \beta(1-s))^2}|x|^2}u(s)\|_{L^2(\mathbb{R}^n)}
\end{equation}
for $s=\frac{\beta t}{\alpha(1-t) + \beta t}.$
\end{lemma} 

\section{Proof of Theorem \ref{theorem 3}}
\subsection{Energy estimate}
We start by proving an energy estimate for solutions of the parabolic equation \eqref{PAR}. As we commented in the introduction, we need to expand the energy estimate in \cite{escauriaza_hardys_2008} to be able to justify all the computations.  
\begin{lemma}\label{lemma 1 b)}
    Let $u\in C([0,1],L^2(\mathbb{R}^n)) \cap L^2([0,1],H^1(\mathbb{R}^n))$ satisfy
    \begin{equation*} \partial_t u = A\Delta u + iB(Lu + V(x,t) u + F(x,t)) \ \ \  \text{in } \ \mathbb{R}^n \times [0,1],\end{equation*}
    for $A>0$, $B\in \mathbb{R}$, $F\in L^2(\mathbb{R}^n\times [0,1]), \ V$ bounded. Then
    \begin{align*}   
   \|e^{s(t)|x|^2} & u(T)\|_{L^2(\mathbb{R}^n)}^2 +A\| \nabla(e^{s(t)|x|^2} u)\|_{L^2(\mathbb{R}^n \times [0,T])}^2  +2A\| 2s(t)|x|e^{s(t)|x|^2} u\|_{L^2( \mathbb{R}^n \times [0,T])}^2 \\ & \leq e^{M_V+|B|} \left(\|e^{\gamma |x|^2} u_0\|_{L^2(\mathbb{R}^n)} +  \|e^{s(t)|x|^2} F(t)\|_{L^2(\mathbb{R}^n \times [0,T])}^2\right),
    \end{align*}
    for $s(t)=\frac{\gamma A}{A+8(A^2+B^2)\gamma t}$, $\gamma>0$, $T\in [0,1]$ and $M_V=\sup_{t\in [0,1]}\|B\text{Im}\ V(t)\|_{L^{\infty}(\mathbb{R}^n)}$.
\end{lemma}
\begin{proof}

Let $f=e^{\phi} u,$ where $\phi=\phi(x,t)$ is a real-valued weight function to be chosen later. 
Then 
$$\partial_t f =  (\mathcal{S+A})f + iBe^\phi F,$$
where 
\begin{align*}
\mathcal{S}&=\partial_t \phi + A(|\nabla \phi|^2 + \Delta) - iB(L\phi + 2\nabla \phi \cdot \nabla_H) -  BIm V, \\
\mathcal{A}&=iB(\nabla \phi \cdot \nabla_H \phi + L) - A(\Delta \phi + 2\nabla \phi \cdot \nabla) + iB Re V ,
\end{align*}
and
\begin{align}\label{dtf}
    \partial_t \|f\|_{L^2(\mathbb{R}^n}^2) &=2 Re \langle Sf, f \rangle_{L^2(\mathbb{R}^n)} + 2Re \langle iB e^\phi F, f\rangle_{L^2(\mathbb{R}^n)}.
\end{align}
Integration by parts shows that
\begin{align*}
    Re\langle Sf, f\rangle &= -A \int_{\mathbb{R}^n} |\nabla f|^2 + \int_{\mathbb{R}^n} \left( A|\nabla \phi|^2 + \partial_t \phi  \right)|f|^2 + 2 B Im \int_{\mathbb{R}^n} \nabla_H \phi \cdot \nabla f  \bar{f} - B\int_{\mathbb{R}^n} Im V|f|^2.
\end{align*}
Cauchy Schwartz' and Young's inequalities show that
\begin{align*}
    2BIm\int_{\mathbb{R}^n}\nabla \phi \cdot \nabla_H f \bar{f} & \leq  \frac{A}{2}\int_{\mathbb{R}^n} |\nabla_H f|^2 + \frac{2B^2}{A}\int_{\mathbb{R}^n} |\nabla \phi f|^2, 
    \end{align*}
and
\begin{align*}
Re\langle Sf, f \rangle &\leq -\frac{A}{2}\int_{\mathbb{R}^n}|\nabla f|^2 + \int_{\mathbb{R}^n}((A+\frac{2B^2}{A})|\nabla \phi|^2 + \partial_t \phi) |f|^2 + \int_{\mathbb{R}^n}|BImV||f|^2 \\ &= -\frac{A}{2}\int_{\mathbb{R}^n}|\nabla f|^2 -A \int_{\mathbb{R}^n} |\nabla \phi f|^2+ \int_{\mathbb{R}^n}((2A+\frac{2B^2}{A})|\nabla \phi|^2 + \partial_t \phi) |f|^2 + \int_{\mathbb{R}^n}|BIm V||f|^2. 
\end{align*}
If $\phi$ is a function that satisfies 
\begin{equation}\label{deriv}
\left(\frac{2B^2}{A}+2A\right)|\nabla \phi|^2 + \partial_t \phi \leq 0 \quad \text{in } \mathbb{R}^n\times [0,1],
\end{equation}
then
\begin{align}\label{derivf}
\partial_t \nonumber\|f\|_{L^2(\mathbb{R}^n)}^2&\leq -A \|\nabla f\|_{L^2(\mathbb{R}^n)}^2-2A\|\nabla \phi f\|_{L^2(\mathbb{R}^n)}^2 +2\|BIm V(t)\|_{L^\infty(\mathbb{R}^n)}\|f\|_{L^2(\mathbb{R}^n)}^2 \\  & \quad +2|B|\|e^{\phi}F\|_{L^2(\mathbb{R}^n)}\|f\|_{L^2(\mathbb{R}^n)}.
\end{align}
Ideally, we want $\phi(x,t)=s(t)|x|^2,$ which satisfies \eqref{deriv} when $s$ solves the IVP 
\begin{equation}
    s'(t) + 8\left(\frac{B^2}{A} +A \right)s(t)^2=0, \quad s(0)=\gamma,
\end{equation}
This holds when $s(t)=\frac{\gamma A}{8(A^2+B^2)\gamma t + A}.$
Let $M_V = \sup_{t\in [0,1]} \|BIm V(t)\|_{L^\infty(\mathbb{R}^n)}$.
Integrating \eqref{derivf} from $0$ to $t$ and Young's inequality implies that
\begin{align*}
\|f\|_{L^2(\mathbb{R}^n)}^2 \leq  &\|f(0)\|_{L^2(\mathbb{R}^n)}^2 -A\int_{0}^{t} \|\nabla f\|_{L^2(\mathbb{R}^n)}^2 ds -2A\int_{0}^{t}\|\nabla \phi f\|_{L^2(\mathbb{R}^n)}^2 ds +|B|  \int_{0}^{t}\|e^{\phi} F \|_{L^2(\mathbb{R}^n)}^2 ds \\& + (M_V + |B|)  \int_{0}^{t} \|f\|_{L^2(\mathbb{R}^n)}^2 ds.
\end{align*}
Applying Grönwall's integral inequality, we deduce that
\begin{align*}
    \|f\|_{L^2(\mathbb{R}^n)}^2 \leq \psi(t) + \int_{0}^{t} \psi (s) (M_{V}+|B|) e^{(M_V + |B|)(t-s)}ds,
\end{align*}
where $$\psi(t) = -A\int_{0}^{t} \|\nabla f\|_{L^2(\mathbb{R}^n)}^2 ds -2A \int_{0}^{t} \|\nabla \phi f\|_{L^2(\mathbb{R}^n)}^2 ds + \|f(0)\|_{L^2(\mathbb{R}^n)} +  |B |\int_{0}^{t} \|e^{\phi}F\|_{L^2(\mathbb{R}^n)}^2 ds.$$ 
Moreover, we have that
\begin{align*}
   & (M_V + |B|) \int_0^{t} \psi(s)  e^{(M_V+|B|)(t-s)} ds \\ & \leq \|f(0)\|_{L^2(\mathbb{R}^n)}^2 (e^{(M_V + |B|)t} - 1) + |B|\int_{0}^{t} \|e^{\phi}F\|_{L^2(\mathbb{R}^n)}^2 d\tau (e^{(M_V+|B|)t}-1) \\ & \ \ -A\int_{0}^{t} \|\nabla f\|_{L^2(\mathbb{R}^n)}^2 d\tau (e^{(M_V+|B|)t}-1) - 2A\int_{0}^{t} \|\nabla \phi f\|_{L^2(\mathbb{R}^n)}^2 d\tau (e^{(M_V+|B|)t}-1).
    \end{align*}
Thus,
\begin{align*}
\|f\|_{L^2(\mathbb{R}^n)}^2 &\leq e^{(M_V+|B|)t}\Big(-A \int_{0}^t \|\nabla f\|_{L^2(\mathbb{R}^n)} ds -2A \int_{0}^t \|\nabla\phi f\|_{L^2(\mathbb{R}^n)} ds + \|f(0)\|_{L^2(\mathbb{R}^n)} \\ &\quad + |B|\int_{0}^t \|e^{\phi}F\|_{L^2(\mathbb{R}^n)} ds\Big), 
\end{align*}
which implies 
\begin{align*}
\|f\|_{L^2(\mathbb{R}^n)}^2 + A\|\nabla f\|_{L^2( \mathbb{R}^n \times [0,t])}^2 + 2A\|\nabla\phi f\|_{L^2( \mathbb{R}^n \times [0,t])}^2 &\leq e^{(M_V +|B|)} \left(\|f(0)\|_{L^2(\mathbb{R}^n)}  +\|e^{\phi}F\|^2_{L^2( \mathbb{R}^n \times [0,t])}\right).
\end{align*}
This proves the result formally in the case when $\phi(x,t)=s(t)|x|^2.$
To prove the result rigorously, we follow the argument in \cite{escauriaza_hardys_2008}. We define a cutoff function
\begin{align*}
\psi_R(x)=
 \begin{cases}
        |x|^2 & \text{if} \ \ \ |x|\leq R, \\
        R^2 & \text{if} \ \ \ |x|>R,
    \end{cases},
\end{align*}
a radial mollifier $\theta_\rho$ with integral 1, and set $\phi_{\rho,R}(x,t) = s(t)(\theta_\rho*\psi_R)(x),$ for $s(t) =\frac{\gamma A}{8(A^2+B^2)t\gamma + A}.$

It follows that  \begin{align*}
    \|f_{\rho, R}(t)\|_{L^2(\mathbb{R}^n)}\leq e^{\gamma R^2}\|u(t)\|_{L^2(\mathbb{R}^n)},
\end{align*}
so that $f_{\rho,R}\in L^2(\mathbb{R}^n)$ for all $t\in [0,1].$ Similarly, since 
$$\nabla f_{\rho, R} = e^{\phi_{\rho,R}}(\nabla \phi_{\rho,R}u +\nabla u),$$ and
$|\nabla \phi_{\rho,R}|\leq CR,$
we also have 
$$\|\nabla f_{\rho,R}\|_{L^2(\mathbb{R}^n)} \leq e^{\gamma R^2}\left(C\gamma R \|u\|_{L^2(\mathbb{R}^n)} + \|\nabla u\|_{L^2(\mathbb{R}^n)}\right) < \infty.$$
Moreover, 
\begin{align*}
    \partial_t\phi_{\rho,R} + \left( \frac{2B^2}{A}+2A\right)|\nabla \phi_{\rho,R}|^2 &=s'(t)(\theta_\rho * \psi_R) +\left( \frac{2B^2}{A}+2A\right)|\theta_\rho *\nabla \psi_R|^2 \\
    &\leq s'(t)(|x|^2+C(n)\rho^2) + s(t)^2\left( \frac{2B^2}{A}+2A\right)(C(n)(|x|^2+\rho^2),
\end{align*}
so that
\begin{align*}
  \partial_t\phi_{\rho,R} +8\left( \frac{2B^2}{A}+2A\right)|\nabla \phi_{\rho,R}|^2 & \leq s'(t)(|x|^2+C(n)\rho^2) + 8s(t)^2\left( \frac{B^2}{A} +A\right)C(n)(|x|^2 + \rho^2) \\ &\leq \left(s'(t)+8\left(\frac{B^2}{A} + A\right) s(t)^2)\right)C(n)(|x|^2+\rho^2)  =0,
\end{align*}
and $\phi_{\rho,R}$ satisfies \eqref{deriv} uniformly in $R$ and $\rho.$
Thus, by applying the formal argument for $f_{\rho,R}$ the argument follows as we let $\rho$ to $0$ and $R$ to $\infty.$
\end{proof}

\subsection{Carleman estimates}
Then we state a general Carleman estimate proven in \cite{escauriaza_hardys_2008}.
\begin{lemma} \label{lemma2}
    Suppose that $\mathcal{S}$ is a symmetric operator, $\mathcal{A}$ is a skew-symmetric operator, both can depend on a time variable, $G$ is a positive function, $f(x,t)$ is a reasonable function
    $$H(t) = \langle f, f \rangle_{L^2(\mathbb{R}^n)}, \ \ \ \ \ D(t) = \langle \mathcal{S}f, f \rangle_{L^2(\mathbb{R}^n)}, \ \ \ \ \  N(t) = \frac{D(t)}{H(t)}.$$
    Then
    \begin{align} \label{Lemma 2 statement 1}
        \partial_t^2 H &= 2\partial_t \text{Re} \, \langle \partial_t f - \mathcal{S} f -\mathcal{A} f, f \rangle_{L^2(\mathbb{R}^n)} + 2\langle \partial_t \mathcal{S} f + [\mathcal{S,A}]f, f\rangle_{L^2(\mathbb{R}^n)} \nonumber \\ &+\|\partial_t f - \mathcal{A}f +\mathcal{S}f\|_{L^2(\mathbb{R}^n)}^2 -\|\partial_t f - \mathcal{A}f -\mathcal{S}f\|_{L^2(\mathbb{R}^n)}^2,
    \end{align} and \begin{equation} \label{lemma 2 statement 2}
    N'(t) \geq \frac{\langle \partial_t {\mathcal S} f + [\mathcal{S,A}]f, f\rangle_{L^2(\mathbb{R}^n)} }{H} - \frac{\|\partial_t f - \mathcal{A}f -\mathcal{S} f\|_{L^2(\mathbb{R}^n)}^2}{2H}.
\end{equation}
Moreover, if 
\begin{equation} \label{lemma 2 criteria} |\partial_t f - \mathcal{A}f - \mathcal{S}f|\leq M_1|f| + G \ \ \ \text{in} \ \ \mathbb{R}^n \times [0,1], \ \ \ \partial_t\mathcal{S} + [\mathcal{S,A}] \geq - M_0,\end{equation}
and
$$M_2=\sup_{t\in [0,1]}\left\|\frac{G(t)}{f(t)}\right\|_{L^2(\mathbb{R}^n)} <\infty,$$
then $\log{H(t)}$ is convex in $[0,1]$ and there is a universal constant $N$ s.t

\begin{equation} \label{lemma 2 st. 3}
    H(t) \leq e^{N(M_0 +M_1 +M_2 + M_1^2 + M_2^2)}H(0)^{1-t} H(1)^{t} \ \ \ \ \text{when} \ 0\leq t\leq 1.
\end{equation}
\end{lemma}
\bigskip

Now, we apply this Carleman estimate to the equation \eqref{PAR}.

\begin{lemma}\label{Lemma3}
Let $u\in C([0,1],L^2(\mathbb{R}^n)) \cap L^2([0,1], H^1(\mathbb{R}^n))$ be a solution to the equation
\begin{equation}
    \partial_t u = A \Delta u + iB(Lu + Vu + F(x,t)), \ \ \ \text{in }\mathbb{R}^n \times [0,1]
\end{equation}
for $A>0, B\in \mathbb{R}$, $V$ a complex-valued potential, $\gamma>0,$ and $\sup_{t\in[0,1]} \|V(t)\|_{L^\infty(\mathbb{R}^n)} \leq M_1.$ Set $$M_2 = \sup_{t\in[0,1]} \frac{ \|e^{\gamma|x|^2}F(t)\|_{L^2(\mathbb{R}^n)}}{\|u(t)\|_{L^2(\mathbb{R}^n)}},$$ and assume that $\|e^{\gamma|x|^2}u(1)\|_{L^2(\mathbb{R}^n)}, \|e^{\gamma |x|^2} u_0\|_{L^2(\mathbb{R}^n)}$ and $M_2$ are finite. Then $\|e^{\gamma |x|^2 }u(t)\|_{L^2(\mathbb{R}^n)}$ is logarithmically convex in $[0,1],$ and there is a universal $N>0$ such that
\begin{align*}
\|e^{\gamma |x|^2}u(t)\|_{L^2(\mathbb{R}^n)} \leq e^{N(B^2(M_1^2+M_2^2) + |B|(M_1+M_2)}  \|e^{\gamma |x|^2 }u_0\|_{L^2(\mathbb{R}^n)}^{1-t}\|e^{\gamma |x|^2}u(1)\|_{L^2(\mathbb{R}^n)}^{t}.
\end{align*}
for $0\leq t \leq 1$.
\end{lemma}
\begin{proof}
Define $f=e^{\phi }u.$ Now $f$ satisfies the equation
$$\partial_t f = (\mathcal{S}+\mathcal{A})f + i(Vf + e^{\phi}F),$$
where \begin{align*}
\mathcal{S}&=\partial_t \phi + A(|\nabla \phi|^2 + \Delta) - iB(L\phi + 2\nabla \phi \cdot \nabla_H)=:\partial_t \phi + A {E}_1 -iBH_2  \\ \mathcal{A}&= iB(\nabla \phi \cdot \nabla_H \phi + L) - A(\Delta \phi + 2\nabla \phi \cdot \nabla):=iBH_1 -A E_2.\end{align*}  

It follows that 
\begin{align*}[\mathcal{S,A}] & = -A^2(E_1E_2 -E_2E_1) -B^2(H_1H_2-H_2H_1) +iAB (E_1H_1 - H_1E_1 + H_2E_2 - E_2H_2)  \\ & \ \ \ +A(E_2 \partial_t \phi -\partial_t \phi E_2) + iB(\partial_t \phi H_1 -H_1 \partial_t \phi),
\end{align*}
and by a direct computation we have
\begin{align*}[\mathcal{S,A}] &= A^2(4\nabla \phi \cdot D^2\phi \nabla \phi - 4\nabla \cdot (D^2 \phi \nabla) - \Delta^2 \phi) + B^2(4\nabla \phi \cdot D^2_H \phi \nabla \phi -4 \nabla \cdot D^2_H \phi \nabla - L^2 \phi)\\ & \ \ \  + iAB(E_1H_1 - H_1E_1 + H_2E_2 - E_2 H_2) + 2A(\nabla \phi \cdot \nabla \partial_t \phi) -iB(2\nabla \partial_t \phi \cdot \nabla_H + L(\partial_t \phi)).
\end{align*}

Since 
\begin{equation*}
\partial_t \mathcal{S} = \partial_t ^2 \phi + 2A(\nabla \phi \cdot \nabla \partial_t \phi) -iB(2\nabla \partial_t \phi \cdot \nabla_H + L (\partial_t \phi)),
\end{equation*}
\begin{align}\label{commutator}
\partial_t \mathcal{S} + [\mathcal{S,A}] &=  \nonumber \partial_t ^2 \phi + 4A(\nabla \phi \cdot \nabla \partial_t \phi) -2iB(2\nabla \partial_t \phi \cdot \nabla_H + L (\partial_t \phi)) \\ & \nonumber \ \ \ + A^2(4\nabla \phi \cdot D^2\phi \nabla \phi - 4\nabla \cdot (D^2 \phi \nabla) - \Delta^2 \phi) + B^2(4\nabla \phi \cdot D^2_H \phi \nabla \phi -4 \nabla \cdot D^2_H \phi \nabla - L^2 \phi)\\ & \ \ \  + iAB([E_1,H_1] -[E_2,H_2]).
\end{align}

The commutator $[\mathcal{S,A}]$ differs from the elliptic case, since $E_j\neq H_j$ for $j=1,2.$ In particular, the cross terms between $Ej$ and $H_j$ that appears here, cancels out in the elliptic case. However, in the hyperbolic case we need to compute the commutators $[E_1,H_1]$ and $[E_2,H_2]$. We have that
\begin{align}
[E_1,H_1] f  & = \Delta(\nabla \phi \cdot \nabla_H \phi) f + 2 \nabla (\nabla \phi \cdot \nabla_H \phi) \cdot \nabla f - L(|\nabla \phi|^2) f - 2 \nabla_H(|\nabla \phi|^2)\cdot \nabla f,
\end{align}
\begin{align}
[H_2,E_2]f & = 2 \nabla_H \phi \cdot \nabla (\Delta \phi) f + 4\nabla_H \phi \cdot \nabla(\nabla \phi \cdot \nabla f) - 2\nabla \phi \cdot \nabla(L\phi) f - 4\nabla \phi \cdot \nabla(\nabla_H \phi \cdot \nabla f). 
\end{align}
Observe that both commutators depend on the mixed-partial derivatives of $\phi.$ For example,
\begin{align*}
\Delta(\nabla \phi \cdot \nabla_H \phi) -L(|\nabla \phi|^2)&= \nabla\cdot \nabla( \nabla \phi \cdot \nabla_H \phi) - \nabla \cdot \nabla_H(\nabla \phi \cdot \nabla \phi) \\ & =2 \nabla \cdot ( D^2 \phi \nabla_H \phi) - 2\nabla \cdot(D^2_H \phi \nabla_H \phi).
\end{align*}
If the mixed partial derivatives of $\phi$ are zero, then $D^2\phi = D^2_H\phi$, so that
\begin{align*}
    \Delta(\nabla \phi \cdot \nabla_H \phi)f -L(|\nabla \phi|^2)f &= 2\nabla \cdot (D^2\phi \nabla_H \phi) -2\nabla \cdot (D^2 \phi \nabla_H \phi) =0.
\end{align*}
Similarly, if all mixed partial derivatives are zero, then
\begin{align*}
    (2\nabla(\nabla \phi \cdot \nabla_H \phi) - 2\nabla_H(\nabla \phi \cdot \nabla \phi)) \cdot \nabla f &=0, \\
    (\nabla_H \phi \cdot(\nabla \phi)- \nabla \phi \cdot \nabla(L\phi))f &= 0, \\
    4\nabla_H \phi \cdot \nabla(\nabla \phi \cdot \nabla f) - 4\nabla \phi \cdot \nabla(\nabla_H \phi \cdot \nabla f) &= 0
\end{align*}
and in this case
\begin{align}\label{commutator special case}
\partial_t \mathcal{S} + [\mathcal{S,A}] &=  \nonumber \partial_t ^2 \phi + 4A(\nabla \phi \cdot \nabla \partial_t \phi) -2iB(2\nabla \partial_t \phi \cdot \nabla_H + L (\partial_t \phi)) \\ &  \ \ \ + (A^2+B^2)(4\nabla \phi \cdot D^2\phi \nabla \phi - 4\nabla \cdot (D^2 \phi \nabla)) -A^2 \Delta^2 \phi+ - B^2L^2 \phi.
\end{align}

If $\phi=\gamma|x|^2,$ then
$$\langle \partial_t \mathcal{S} f + [\mathcal{S,A}] f , f \rangle_{L^2(\mathbb{R}^n)} = \gamma(A^2+B^2)\int_{\mathbb{R}^n}32 \gamma^2 |x|^2 |f|^2 + 8|\nabla f|^2 \geq 0.$$
Since also
$$|\partial_t f - \mathcal{S}f-\mathcal{A}f|\leq |B| (\sup_{t\in [0,1]}\|V\|_{L^\infty(\mathbb{R}^n)}|f| + e^{\phi}|F|),$$
it follows formally by Lemma \ref{lemma2} that
\begin{align*}
   \|e^{\gamma |x|^2} u(t)\|_{L^2(\mathbb{R}^n)}\leq e^{N(B^2(M_1^2 +M_2^2) + |B|(M_1+M_2)) } \|e^{\gamma |x|^2}u_0\|_{L^2(\mathbb{R}^n)}^{1-t}\|e^{\gamma |x|^2}u(1)\|_{L^2(\mathbb{R}^n)}^t.
\end{align*}
However, in order to justify that $\|e^{\phi}u(t)\|_{L^2(\mathbb{R}^n)}$ is finite for all $t\in [0,1]$ we need to modify the weight, which makes the computations more involved. The argument follows the same spirit as in \cite{escauriaza_hardys_2008}. However, due to the presence of the commutators $[E_j,H_j]$ we work with a different weight function, chosen so that all mixed partial derivatives vanish, and therefore $[E_j,H_j]=0$ for this specific weight.

In particular, we start by defining a binary sequence $b=b_1b_2,\dots b_n$, $b_j\in\{0,1\}$, and let
\begin{align*}
R_b:=\Bigg\{x\in \mathbb{R}^n \ \Bigg|\begin{cases}
    |x_j| < 1 & if \ b_j=0\\
    |x_j|\geq 1 &  if \ b_j =1,
\end{cases} \Bigg\}
\end{align*}

\begin{align*}
\psi_a(x_j) :=
 \begin{cases}
x_j^2  & if \ |x_j|<1 \\
\frac{2|x_j|^{2-a}-a}{2-a} & if \ |x_j|\geq 1,
 \end{cases}
\end{align*}

\begin{align*}
\phi_a(x) := \sum_{j=1}^{n}\psi_a(x_j).
\end{align*}
Then $\phi_a$ is defined differently in the $2^n$ regions $R_b$  will be continuous and differentiable in all of $\mathbb{R}^n.$ Moreover,
\begin{align*}
    \nabla \phi_a =(\psi'(x_1),\dots,\psi'(x_n)), \quad 
    \Delta \phi_a =\sum_{j=1}^{n}\psi''_a(x_j), \quad
    L\phi_a = \sum_{j=1}^k\psi''(x_j) - \sum_{j=k+1}^{n} \psi''(x_j),
\end{align*}
where
\begin{align*}
    \psi_a'(x_j) =\begin{cases}
        2x_j & if \ |x_j|<1, \\
        2|x_j|^{1-a}sgn(x_j) & if \ |x_j|\geq 1,
    \end{cases} \quad \quad
    \psi_a''(x_j)=\begin{cases}
        2 & if \ |x_j|<1, \\
        2(1-a)|x_j|^{-a} & |x_j|\geq 1.
    \end{cases}
\end{align*}
Since neither $\Delta \phi_a$ or $L\phi_a$ is continuous at the boundaries of the regions $R_b$, we need to compute the third derivative in the distributional sense. It follows that the distribution
\begin{align*}
    \partial_{x_j} \Delta \phi_a & =  \psi_a(x_j)''' \\ & = 2a(\mu_{(x_j=-1)}-\mu_{(x_j=1)})-2a(a-1)|x_j|^{-a-1}sgn(x_j) \mathbbm{1}_{|x_j|\geq 1},
\end{align*} where
$\mu_{x_j=\pm 1}$ is the distribution defined by $$\langle \mu_{(x_j=\pm1)},g\rangle=\int_{\mathbb{R}^{n-1}} g(x_1,\dots,x_{j-1},\pm 1,x_{j+1},x_n)dx$$ for $g\in C^\infty_0(\mathbb{R}^n).$ Moreover, let $\theta_\rho\in C^\infty_0(\mathbb{R}^n)$ be a mollifier, and let $\phi_{a,\rho}=\gamma \phi_a*\theta_\rho.$ It follows that
$\Delta^2\phi_{a,\rho} = \gamma \sum_{j=1}^n \partial_j\Delta \phi_a*\partial_j \theta_{\rho},$ so that
\begin{align*}
\|\Delta^2\phi_{a,\rho}
\|_{L^\infty(\mathbb{R}^n)} &\leq \gamma \sum_{j=1}^n \|\partial_j \Delta \phi_a *\partial_j \theta_\rho\|_{L^\infty(\mathbb{R}^n)} \\ &\leq \gamma \sum_{j=1}^n\Big(2a(1-a)\||x_j|^{-a-1}\mathbbm{1}_{|x_j|\geq 1}*\partial_j\theta_\rho\|_{L^\infty(\mathbb{R}^n)}  \\ & \quad +  2a\|(\mu_{(x_j=-1)} - \mu_{(x_j=1)})*\partial_j\theta_\rho\|_{L^\infty(\mathbb{R}^n)}\Big).  \end{align*}
Since
$$\|(\mu_{(x_j=-1)} - \mu_{(x_j=1)})*\partial_j\theta_\rho\|_{L^\infty(\mathbb{R}^n)} \leq C(n,\rho,\theta)$$ and  $$\||x_j|^{-a-1}\mathbbm{1}_{|x_j|\geq 1}*\partial_j \theta_\rho\|_{L^\infty(\mathbb{R}^n)} \leq \|1*\partial_j \theta_\rho \|_{L^\infty(\mathbb{R}^n)}\leq C(n,\rho,\theta),$$
it follows that
\begin{equation}\label{Bound for D}
\|\Delta^2 \phi_{a,\rho}\|_{L^\infty(\mathbb{R}^n)} \leq C(n,\rho, \theta)a.
\end{equation}
A similar argument shows that
\begin{equation}\label{Bound for L}\|L^2 \phi_{a,\rho}\|_{L^\infty(\mathbb{R}^n)} \leq C(n,\rho, \theta)a. \end{equation}
Since the mixed partial derivatives are 0, $D^2\phi_{a,\rho} = D^2_{H} \phi_{a\rho}$, and 
(\ref{commutator}) now becomes 
\begin{align} \label{commutator for phi_a,rho}
\partial_t \mathcal{S}_{a,\rho} + [\mathcal{S}_{a,\rho},\mathcal{A}_{a,\rho}] & = 4(A^2+B^2)(\nabla_{\phi_{a,\rho}}\cdot D^2\phi_{a,\rho}\nabla \phi_{a,\rho} - \nabla \cdot D^2\phi_{a,\rho} \nabla) -A^2 \Delta^2 \phi_{a,\rho} -  B^2L^2 \phi_{a,\rho},
\end{align}
and so
 \begin{align} \label{St+SA}
    \langle (\partial_t \mathcal{S}_{a,\rho} + [\mathcal{S}_{a,\rho},\mathcal{A}_{a,\rho}] )f_{a,\rho},f_{a,\rho}\rangle&= \nonumber 4(A^2+B^2)\left( \langle D^2 \phi_{a,\rho} \nabla \phi_{a,\rho} \cdot \nabla \phi_{a,\rho} f_{a,\rho}, f_{a,\rho}\rangle - \langle \nabla \cdot (D^2 \phi_{a,\rho} \nabla f_{a,\rho}), f_{a,\rho} \rangle  \nonumber  \right)\\ & \quad  - A^2 \langle \Delta^2 \phi_{a,\rho} f_{a,\rho}, f_{a,\rho} \rangle -  B^2 \langle L^2 \phi_{a,\rho} f_{a,\rho}, f_{a,\rho} \rangle  \nonumber  \\ & = 4(A^2+B^2)\left( \langle \nabla \phi_{a,\rho}^T D^2 \phi_{a,\rho} \nabla \phi_{a\rho} f_{a,\rho}, f_{a,\rho}\rangle +\langle D^2\phi_{a,\rho} \nabla f_{a,\rho}, \nabla f_{a,\rho} \rangle \right) \nonumber \\ & \quad  -A^2 \langle \Delta^2 \phi_{a,\rho} f_{a,\rho}, f_{a,\rho} \rangle -  B^2\langle L^2 \phi_{a,\rho} f_{a,\rho}, f_{a,\rho} \rangle.
\end{align}
Since $D^2 \phi_a(x) = \text{diag}(\psi_a''(x_1),\dots,\psi_a''(x_n)),$ is positive definite, and by (\ref{Bound for D}) and (\ref{Bound for L}), it follows that
\begin{equation}
\partial_t \mathcal{S}_{a,\rho} + [\mathcal{S}_{a,\rho},\mathcal{A}_{a,\rho}]  \geq -C(\rho,\theta)a =-M_0(a,\rho), 
\end{equation}
where $M_0(a,\rho)\to 0$ when $a\to 0.$ Moreover, 
\begin{equation}\label{dtS SA}
    |(\partial_t -\mathcal{S}_{a,\rho} - \mathcal{A}_{a,\rho})f_{a,\rho}|  \leq |B| (M_1|f_{a,\rho}| + e^{\phi_{a,\rho}}F),
\end{equation}
 and by using that $\phi_{a,\rho} \leq \gamma|x|^2 + C(n)\rho^2, $ we have that \begin{equation}
     M_2(a,\rho):=\sup_{t\in [0,1]} \frac{\|e^{\phi_{a,\rho}}F\|_{L^2(\mathbb{R}^n)}}{\|f_{a,\rho}\|_{L^2(\mathbb{R}^n)}}\leq e^{C(n)\rho^2}M_2:=M_{2,\rho},\end{equation}
Applying Lemma \ref{lemma2}, we get that
$\|f_{a,\rho}\|_{L^2(\mathbb{R}^n)}$ is logarithmically convex for $t\in [0,1]$ and that
\begin{equation}
\|e^{\phi_{a,\rho}}u\|_{L^2(\mathbb{R}^n)} \leq e^{N(B^2(M_1^2 + M_{2,\rho}^2) + |B|(M_1 + M_{2,\rho} + M_0(a,\rho))} \|e^{\phi_{a,\rho}}u_0\|_{L^2(\mathbb{R}^n}^{(1-t)}\|e^{ \phi_{a,\rho}}\|_{L^2(\mathbb{R}^n)}^t.
\end{equation}
By first taking the limit as $a\to 0$, and then the limit as $\rho$ to $0$ we deduce that
\begin{align*}
\|e^{\gamma |x|^2}u(t)\|_{L^2(\mathbb{R}^n)}^2 \leq e^{N(B^2(M_1^2+M_2^2) + |B|(M_1+M_2)}  \|e^{\gamma |x|^2 }u_0\|_{L^2(\mathbb{R}^n)}^{2(1-t)}\|e^{\gamma |x|^2}u(1)\|_{L^2(\mathbb{R}^n)}^{2t}.
\end{align*}
\end{proof}
Now we prove a similar result for $\nabla u.$
\begin{lemma} \label{Lemma 4 paper}
Let $u\in C([0,1],L^2(\mathbb{R}^n)) \cap L^2([0,1], H^1(\mathbb{R}^n))$ be a solution to the equation
\begin{equation}
    \partial_t u = A \Delta u + iB(Lu + Vu + F(x,t)) \ \ \ \text{in }\mathbb{R}^n \times [0,1]
\end{equation}
for $A>0, B\in \mathbb{R}$, $V$ a complex-valued potential, $\gamma>0,$ and $\sup_{t\in[0,1]} \|V(t)\|_{L^\infty(\mathbb{R}^n)} \leq M_1.$
   Then for $0<t<1$
    \begin{align} \label{Bound lemma4}
    \nonumber \|\sqrt{t(1-t)}e^{\gamma|x|^2}\nabla & u\|_{L^2(\mathbb{R}^n\times [0,1])} +  \|\sqrt{t(1-t)}|x|e^{\gamma |x|^2}u\|_{L^2(\mathbb{R}^n\times [0,1])} \\ &\leq N(A,B, \gamma, M_1) [\sup_{t\in [0,1]} \|e^{\gamma |x|^2}u (t)\|_{L^2(\mathbb{R}^n)} + \sup_{t\in [0,1]} \|e^{\gamma |x|^2} F\|_{L^2(\mathbb{R}^n)}]
\end{align}
where $N(A,B,\gamma, M_1)$ remains finite when $A^2+B^2$ is bounded away from $0$. 
\end{lemma}
\begin{proof}
Let $f=e^{\phi}u.$ Assuming all computations in Lemma \ref{lemma2} are justified for $f$, we start by multiplying (\ref{Lemma 2 statement 1}) with $t(1-t)$ and integrate from 0 to 1. Integrating by parts two times, we get for the left-hand side of (\ref{Lemma 2 statement 1})
   \begin{align} \label{Ibp4.1}
    \int_{0}^{1} \partial_t^2 H(t)t(1-t) dt & \leq H(0) + H(1).
   \end{align}
For the first part of the right-hand side of (\ref{Lemma 2 statement 1}), we get
\begin{align} \label{Ibp4.2}
2\int_{0}^{1} t(1-t)  \partial_t Re \langle \partial_t f - \mathcal{S}f -\mathcal{A}f, f\rangle_{L^2(\mathbb{R}^n)} dt & = -2\int_{0}^{1} (1-2t)  Re \langle \partial_t f -\mathcal{S}f-\mathcal{A}f,f\rangle_{L^2(\mathbb{R}^n)} dt.
\end{align}
It then follows from (\ref{Lemma 2 statement 1}), (\ref{Ibp4.1}) and (\ref{Ibp4.2}) that
\begin{align}\label{h(0)h(1)}
 2\int_{0}^1 t(1-t) \langle \partial_t \mathcal{S} f + [\mathcal{S,A}]f,f &\rangle_{L^2(\mathbb{R}^n)} \nonumber  \\ & \leq H(1) + H(0) + 2\int_{0}^1 (1-2t) Re \langle \partial_t f -\mathcal{S}f-\mathcal{A}f,f\rangle_{L^2(\mathbb{R}^n)} dt \nonumber  \\ & \ \ + \int_{0}^1 t(1-t)  \|\partial_t f - \mathcal{A}f - \mathcal{S}f\|_{L^2(\mathbb{R}^n)}^2dt.
\end{align}  

Ideally, we want $\phi=\gamma |x|^2,$ but to be able to justify everything rigorously, we will work with $\phi_{a,\rho}=\gamma(\phi_a * \theta_\rho)$ where $\phi_a$ is defined just as in Lemma \ref{Lemma3}. Observe that by Lemma \ref{lemma 1 b)}, $\nabla f_{a,\rho}$ and $\nabla \phi_{a,\rho} f_{a,\rho} \in L^2(\mathbb{R}^n \times [0,T])$, and we are able to rigorously justify all computations, 

We start by bounding the left hand side of (\ref{h(0)h(1)}) from below.
Recall from (\ref{St+SA}) that
\begin{align} \label{commutator111}
\langle (\partial_t \mathcal{S}_{a,\rho } + [\mathcal{S,A}])f_{a,\rho},f_{a,\rho} \rangle & =4(A^2+B^2)\left( \langle \nabla \phi_{a,\rho}^T D^2 \phi_{a,\rho} \nabla \phi_{a\rho} f_{a,\rho}, f_{a,\rho}\rangle +\langle D^2\phi_{a,\rho} \nabla f_{a,\rho}, \nabla f_{a,\rho} \rangle \right) \nonumber \\ & \quad -A^2 \langle \Delta^2 \phi_{a,\rho} f_{a,\rho}, f_{a,\rho} \rangle -  B^2\langle L^2 \phi_{a,\rho} f_{a,\rho}, f_{a,\rho} \rangle.
\end{align} Observe that 
\begin{align*}
    \langle D^2\phi_{a,\rho} \nabla f_{a,\rho}, \nabla f_{a,\rho} \rangle_{L^2(\mathbb{R}^n)} 
     =\gamma \langle( D^2 \phi_a - 2I )*\theta_\rho \nabla f_{a,\rho}, \nabla f_{a,\rho} \rangle_{L^2(\mathbb{R}^n)}+ 2\gamma\|\nabla f_{a,\rho}\|_{L^2(\mathbb{R}^n)}^2,
\end{align*} 
$$D^2\phi_a-2I =\text{diag}(\psi_a''(x_1)-2,\dots,\psi_a''(x_n)-2),$$ and $$\psi_a''(x_j)-2=\begin{cases}
    0 & \text{if } |x_j|<1 \\ 2(1-a)|x_j|^{-a} -2 &\text{if }  |x_j|\geq 1,
\end{cases}$$
goes to $0$ pointwise as $a\to 0$. Similarly,
\begin{align*}
 \langle \nabla \phi_{a,\rho}^T D^2 \phi_{a,\rho} \nabla \phi_{a,\rho}f_{a,\rho}, f_{a,\rho} \rangle_{L^2(\mathbb{R}^n)} =\gamma\langle \nabla \phi_{a,\rho}^T (D^2 \phi_a - 2I)*\theta_\rho \nabla \phi_{a,\rho} f_{a,\rho},f_{a,\rho}\rangle_{L^2(\mathbb{R}^n)} + 2\gamma\|\nabla \phi_{a,\rho} f_{a,\rho} \|_{L^2(\mathbb{R}^n)} ^2,
\end{align*}
and since both $L^2\phi_{a,\rho}\leq C(n,\rho)a$ and $\Delta^2\phi_{a,\rho}\leq C(n,\rho)a,$ 
\begin{align*}
\langle (\partial_t \mathcal{S}_{a,\rho } + [\mathcal{S,A}])f_{a,\rho},f_{a,\rho} \rangle & \geq 4 (
A^2+B^2) \gamma \Big( \langle( D^2 \phi_a - 2I )*\theta_\rho \nabla f_{a,\rho}, \nabla f_{a,\rho} \rangle_{L^2(\mathbb{R}^n)}+ 2\|\nabla f_{a,\rho}\|_{L^2(\mathbb{R}^n)}^2 \\ & \quad + \langle \nabla \phi_{a,\rho}^T (D^2 \phi_a - 2I)*\theta_\rho \nabla \phi_{a,\rho} f_{a,\rho},f_{a,\rho}\rangle_{L^2(\mathbb{R}^n)} + 2\|\nabla \phi_{a,\rho} f_{a,\rho} \|_{L^2(\mathbb{R}^n)} ^2\Big) \\ & \quad - C(n,\rho)a\|f_{a,\rho}\|_{L^2(\mathbb{R}^n)}^2.
\end{align*}

Furthermore, since
$$\nabla f_{a,\rho} = e^{\phi_{a,\rho}}(\nabla \phi_{a,\rho} u + \nabla u), $$
\begin{align*}
\int_{\mathbb{R}^n} |\nabla f_{a,\rho}|^2 dx = & \int_{\mathbb{R}^n} e^{2\phi_{a,\rho}}(|\nabla \phi_{a,\rho}|^2 |u|^2 + |\nabla u|^2) dx + \int_{\mathbb{R}^n} e^{2\phi_{a,\rho}}( \nabla \phi_{a,\rho} u \cdot \nabla \bar{u}) dx \\ &  + \int_{\mathbb{R}^n} e^{2\phi_{a,\rho}}( \nabla  u \cdot \nabla \phi_{a,\rho} \bar{u}) dx.
\end{align*}
Integrating by parts shows that
\begin{align*}
\int_{\mathbb{R}^n} e^{2\phi_{a,\rho}} \nabla \phi_{a,\rho} u \cdot \nabla \bar{u} dx = & -\int_{\mathbb{R}^n} e^{2\phi_{a,\rho}} \nabla \phi_{a,\rho} \cdot \nabla u \bar{u} dx -\int_{\mathbb{R}^n} e^{2\phi_{a,\rho}} \nabla \cdot (\nabla \phi_{a,\rho}) |u|^2 dx  \\ & - \int_{\mathbb{R}^n} 2 |\nabla \phi_{a,\rho}|^2 e^{2\phi_{a,\rho}} |u|^2 dx,
\end{align*}
so that 
\begin{align}\label{EQ1}
\int_{\mathbb{R}^n} |\nabla f_{a,\rho}|^2 + |\nabla \phi_{a,\rho}|^2 |f_{a,\rho}|^2 dx  & = \int_{\mathbb{R}^n} e^{2\phi_{a,\rho}} |\nabla u|^2 dx -\int_{\mathbb{R}^n} |f_{a,\rho}|^2 \nabla \cdot (\nabla \phi_{a,\rho}) dx.
\end{align}
In addition, integration by parts, Cauchy-Schwarz and Young's inequalities again show that
\begin{align}\label{EQ2}\int_{\mathbb{R}^n} |f_{a,\rho}|^2 \nabla \cdot (\nabla \phi_{a,\rho})  dx \nonumber &\leq 2 \int_{\mathbb{R}^n} |\nabla f_{a,\rho}| |f_{a,\rho}| |\nabla \phi_{a,\rho}| dx \\ &\leq \int_{\mathbb{R}^n} |\nabla f_{a,\rho}|^2 dx + \int_{\mathbb{R}^n} |f_{a,\rho}|^2 |\nabla \phi_{a,\rho}|^2 dx.
\end{align}
Combining (\ref{EQ1}) and (\ref{EQ2}),
\begin{align*}
    2\int_{\mathbb{R}^n} |\nabla f_{a,\rho}|^2 + |f_{a,\rho}|^2|\nabla \phi_{a,\rho}|^2 dx \geq \int_{\mathbb{R}^n} e^{2\phi_{a,\rho}}|\nabla u|^2 dx,
\end{align*} 
and
\begin{align*}
    2\int_{\mathbb{R}^n} |\nabla f_{a,\rho}|^2 + |f_{a,\rho}|^2|\nabla \phi_{a,\rho}|^2 dx \geq \frac12\int_{\mathbb{R}^n} e^{2\phi_{a,\rho}}|\nabla u|^2 dx + \int_{\mathbb{R}^n}|f_{a,\rho}|^2|\nabla \phi_{a,\rho}|^2 dx.
\end{align*} 
Thus 
\begin{align*}
    &\int_{0}^{1} \langle \partial_t \mathcal{S}_{a,\rho} + [\mathcal{S}_{a,\rho},\mathcal{A}_{a,\rho}] f_{a,\rho}, f_{a,\rho}\rangle_{L^2(\mathbb{R}^n)} t(1-t) dt \\ & \geq 4 (A^2+B^2) \gamma \Big(  2\int_{0}^1 \int_{\mathbb{R}^n} |\nabla f_{a,\rho}|^2 + |\nabla \phi_{a,\rho}|^2|f_{a,\rho}|^2 t(1-t) dx dt  \\ & \ \ + \int_{0}^{1} \langle (D^2 \phi_{a} - 2I)*\theta_\rho  \nabla f_{a,\rho}, \nabla f_{a,\rho} \rangle_{L^2(\mathbb{R}^n)} +
\langle \nabla \phi_{a,\rho}^T ((D^2\phi_{a}-2I)*\theta_\rho)\nabla \phi_{a,\rho} f_{a,\rho}, f_{a,\rho} \rangle_{L^2(\mathbb{R}^n)} dt \Big)
\\ & \ \ -  \int_{0}^1 C(n,\rho)a \|f_{a,\rho}\|_{L^2(\mathbb{R}^n)}^2 dt  \\ &  \ \geq N(A^2+B^2) \gamma \Big(\int_0^1 \int_{\mathbb{R}^n} t(1-t) e^{2\phi_{a,\rho}}|\nabla u|^2 dx dt + \int_{0}^1 \int_{\mathbb{R}^n} |\nabla \phi_{a,\rho}|^2|f_{a,\rho}|^2 t(1-t) dx dt \\ &  \ \ + \int_{0}^{1} \langle (D^2 \phi_{a} - 2I)*\theta_\rho  \nabla f_{a,\rho}, \nabla f_{a,\rho} \rangle_{L^2(\mathbb{R}^n)}t(1-t)dt \\ & \ \  + \int_{0}^1 \langle\nabla \phi_{a,\rho}^T ((D^2 \phi_{a} - 2I)*\theta_\rho)\nabla \phi_{a,\rho}f_{a,\rho}, f_{a,\rho}\rangle_{L^2(\mathbb{R}^n)} t(1-t) dt\Big)\\ & \ \ - \int_{0}^1 C(n,\rho)a \|f_{a,\rho}\|_{L^2(\mathbb{R}^n)}^2 t(1-t)dt.
\end{align*}
Going back to \eqref{h(0)h(1)}, we deduce that
\begin{align*}
    & (A^2+B^2)\gamma \left(\int_{0}^{1} t(1-t) \int_{\mathbb{R}^n} e^{2\phi_{a,\rho}} |\nabla u|^2 dx dt + \int_{0}^{1}t(1-t) \int_{\mathbb{R}^n} e^{2\phi_{a,\rho}} |u|^2 |\nabla \phi_{a,\rho}|^2 dx dt \right) \\ & \leq N\Big( H(1) + H(0) + 2\int_{0}^1 (1-2t) Re \langle \partial_t f_{a,\rho} -\mathcal{S}_{a,\rho}f_{a,\rho}-\mathcal{A}_{a,\rho}f_{a,\rho},f_{a,\rho}\rangle_{L^2(\mathbb{R}^n)} dt \nonumber  \\ & \ \ + \int_{0}^1 t(1-t)  \|\partial_t f_{a,\rho} - \mathcal{A}_{a,\rho}f_{a,\rho} - \mathcal{S}_{a,\rho}f_{a,\rho}\|_{L^2(\mathbb{R}^n)}^2dt  \\ &  \ \ -  \int_{0}^{1} \langle (D^2 \phi_{a} - 2I)*\theta_\rho  \nabla f_{a,\rho}, \nabla f_{a,\rho} \rangle_{L^2(\mathbb{R}^n)} t(1-t) dt  \\ & \ \ - \int_{0}^1
\langle \nabla \phi_{a,\rho}^T ((D^2 \phi_{a} - 2I)*\theta_\rho)\nabla \phi_{a,\rho}f_{a,\rho}, f_{a,\rho}\rangle_{L^2(\mathbb{R}^n)} t(1-t) dt \\ & \ \ + C(n,\rho)a  \int_{0}^1 \|f_{a,\rho}\|_{L^2(\mathbb{R}^n)}^2 dt \Big),
\end{align*}
for some $N>0.$ Moreover, using using \eqref{dtS SA}, 
\begin{align*}
    \int_{0}^{t}t(1-t)\|\partial_tf_{a,\rho}-\mathcal{S}_{a,\rho} f_{a,\rho}-\mathcal{A}_{a,\rho}f_{a,\rho}\|_{L^2(\mathbb{R}^n)}^2dt \leq N(B,M_1)\sup_{t\in[0,1]} \|f_{a,\rho}\|_{L^2(\mathbb{R}^n)}^2 + \|e^{\phi_{a\,\rho}} F\|_{L^2(\mathbb{R}^n)}^2 ),
\end{align*}

\begin{align*}
    \int_{0}^1 (1-2t) Re \langle \partial_t f_{a,\rho} -\mathcal{S}_{a,\rho}f_{a,\rho}-\mathcal{A}_{a,\rho}f_{a,\rho},f_{a,\rho}\rangle_{L^2(\mathbb{R}^n)} dt & \leq N(B,M_1)\sup_{t\in [0,1]}\left(\|f_{a,\rho}\|^2 + \|e^{ \phi_{a,\rho}}F\|_{L^2(\mathbb{R}^n)}^2\right),
\end{align*}
and since 
$$H(0)+H(1)\leq 2 \sup_{t\in [0,1]}\|e^{ \phi_{a,\rho}}u(t)\|_{L^2(\mathbb{R}^n)}^2,$$
it follows that
\begin{align}\label{limit}
   & \|\sqrt{t(1-t)}e^{\phi_{a,\rho}}\nabla u\|_{L^2(\mathbb{R}^n\times [0,1])}^2 + \|\sqrt{t(1-t)}e^{\phi_{a,\rho}}u\nabla \phi_{a,\rho}\|_{L^2(\mathbb{R}^n \times [0,1])}^2 \nonumber\\ & \leq N(A,B,\gamma, M_1) \Big(\sup_{t\in [0,1]}\|e^{\phi_{a,\rho}}u\|_{L^2(\mathbb{R}^n)} ^2+\sup_{t\in[0,1]}\|e^{\phi_{a,\rho}}F\|_{L^2(\mathbb{R}^n)}^2 \nonumber\\ & \ \ - \int_{0}^{1} \langle (D^2\phi_{a}-2I)*\theta_\rho  \nabla f_{a,\rho}, \nabla f_{a,\rho} \rangle_{L^2(\mathbb{R}^n)} +
\langle \nabla\phi_{a,\rho}^T ((D^2\phi_{a}-2I)*\theta_\rho)\nabla \phi_{a,\rho}f_{a,\rho}, f_{a,\rho}\rangle_{L^2(\mathbb{R}^n)} dt \nonumber \\ & \ \ + C(n,\rho)a  \int_{0}^1 \|f_{a,\rho}\|_{L^2(\mathbb{R}^n)}^2 dt\Big),
\end{align}
and $N(A, B,\gamma,M_1)$ remains bounded when $A^2+B^2$ is bounded away from $0$. The result follows by first letting $a\to 0$ and then $\rho\to 0,$ which also can be rigorously justified by using Lemma \ref{lemma 1 b)}.
\end{proof}

\subsection{\texorpdfstring{$L^2$-weighted bound under a smallness assumption on the potential}{L2-weighted bound under a smallness assumption}}

\begin{lemma}\label{general KPV lemma}
There exists $\epsilon_0>0$ such that if $V:\mathbb{R}^n\times[0,1]\to \mathbb{C}$ satisfies \begin{equation} \label{Vsmall}
    \|V\|_{L^1_t,L^\infty_X}\leq \epsilon _0 
\end{equation}
and $u\in C([0,1],L^2(\mathbb{R}^n))$ is a solution of the IVP
\begin{align}\label{HSVF}
  \begin{cases}
      i\partial_t u+ Lu = Vu +F\\
      u(x,0)=u_0,
      \end{cases}
\end{align}
$F\in L^1([0,1],L^2(\mathbb{R}^n))$ and for some $\beta\in \mathbb{R}$ and some $1\leq j\leq n$ $\|e^{\beta x_j}u_0\|_{L^2(\mathbb{R}^n)}<\infty$, $\|e^{\beta x_j}u(1)\|_{L^2(\mathbb{R}^n)}<\infty$ and $F\in L^1([0,1],L^2(e^{2\beta x_j} dx)),$ then for some $N>0$ independent of $\beta$

\begin{align*}
    \sup_{t\in [0,1]}\|e^{\beta x_j}u(t)\|_{L^2(\mathbb{R}^n)} \leq N(\|e^{\beta x_j}u_0\|_{L^2(\mathbb{R}^n)} + \|e^{\beta x_j}u(1)\|_{L^2(\mathbb{R}^n)} + \|F\|_{L^1([0,1],L^2(\mathbb{R}^n,e^{2\beta x_j}dx))}).
\end{align*}
\end{lemma}

\begin{remark}
The argument follows as in the proof in \cite{kenig_unique_2003}, but we will show the proof in the case of the operator $L$ and for $x_j$ when $k< j\leq n$.
\end{remark}

After the proof of this lemma, we will extend this to show that knowing the decay in each coordinate direction actually implies the same decay bound for  
$
\|e^{\lambda \cdot x}u\|_{L^2(\mathbb{R}^n)}, \lambda \in \mathbb{R}^n.
$
In the elliptic case, since the Laplacian $\Delta$ is invariant under the orthogonal group $O(n)$, it suffices to prove the estimate in just one coordinate direction $x_j$.  
In contrast, for the hyperbolic case we need to have decay estimates for both $1 \leq j \leq k$ and $k < j \leq n$, because $L$ is invariant under the indefinite orthogonal group, (see Subsection \ref{pseudo group}, $O(k,n-k)$.  

\begin{proof}[Proof of Lemma \ref{general KPV lemma}]
Let us suppose $\beta>0$ and define a bump function $\phi_n \in C^\infty_0(\mathbb{R})$ such that $0\leq \phi \leq 1$ and 
\begin{align*}
    \phi_n(s) = \begin{cases}
        1 & s\leq n, \\
        0 & s\geq 10 n,
    \end{cases} \quad \quad
    |\phi_n^{(j)} (s)|\leq \frac{c_j}{n^j}.
\end{align*}
Define also
$$\theta_n(s)  = \beta \int_0^s \phi_n ^2(t) dt.$$ It follows that $\theta_n$ is a smooth non-decreasing function, 
$\theta_n(s) = \beta s$ for $s<n,$ $\theta_n(s) = c_n\beta$ for $s>10n$ and 
$$\theta_n'(s) = \beta \phi_n^2(s) \leq \beta, \quad |\theta_n ^j (s)|\leq \frac{\beta c_j}{n^{j-1}}.$$ Let 
$\psi_n(s) =  e^{2\theta_n(s)}.$ Then $\psi_n(s) \leq e^{2\beta s}$ and $\psi_n(s)\nearrow e^{2\beta s}$ when $n\to \infty.$ Fix $j\in[k+1,n]$ and define $v_n(x,t) = \psi_n(x_j)u(x,t).$ Then

\begin{align*}
    \partial_t v_n & = \psi_n \partial_t u, \\
    \partial x_j v_n  & = 2\beta \phi_n ^2 v_n + \psi_n\partial_{x_j}u \\ 
    L v_n &=\psi_nLu -(4\beta\phi_n'\phi_n-4\beta^2\phi_n^4)v_n -4\beta \phi_n^2\partial_{x_j} v_n.
\end{align*}
Thus 
\begin{align*}
    i\partial_t v_n + Lv_n &= i\psi_n \partial_t u + Lv_n \\ & = Vv_n + \psi_n F -(4\beta \phi_n'\phi_n-4\beta^2\phi_n^4)v_n -4\beta\phi_n^2\partial_{x_j}v_n.
\end{align*}

Since the term $4\beta^2 \phi_n^4 v_n,$ will not be of order $O(1/n)$, we need to get rid of this term. We consider instead the function $w_n(x,t) = e^{\mu(x,t)}v_n(x,t)$ where $\mu(x,t) = +4i\beta^2\phi_n^4(x_1)t.$ 

By a straight forward computation, see also \cite{kenig_unique_2003}, of the derivatives we get that

\begin{align*}
    \partial_t w_n &= e^{\mu}\partial_t v_n + i4\beta^2 \phi_n^4 w_n \\
    \partial_{x_j} w_n & = 16 i\beta^2 \phi_n^3 \phi_n't w_n + e^{\mu}\partial_{x_j}v_n \\
    \partial_{x_j}^2 w_n & = 32i\beta^2 \phi_n^3 \phi_n' t \partial_{x_j} w_n + w_n(48 i \beta^2 \phi_n^2 (\phi_n')^2t + 16i\beta^2 \phi_n^3\phi_n''t + (16i\beta^2\phi_n^3\phi_n't)^2) +e^{\mu} \partial_{x_j}^2v_n,
\end{align*}
and that $w_n$ satisfies the equation
\begin{equation}\label{eq w}
i\partial_t w_n +Lw_n = Vw_n + \tilde{F}_n + h(x,t)w_n -a^2(x_j)\partial_{x_j}w_n + itb(x_j)\partial_{x_j}w_n, 
\end{equation}
where 

\begin{align*}
h(x,t) =& -4\beta \phi_n \phi_n ' - 64\beta^2\phi_n^5 \phi_n't -48i\beta^2\phi_n^2(\phi_n ')^2t - (16i\beta^2\phi_n^3\phi_n''t)^2 \\
a^2(x) &= 4\beta \phi_n^2(x_j), \\
b(x) &= - 32\beta^2\phi_n^3(x_j) \phi_n'(x_j) \\
\tilde{F}_n(x,t)&=e^{\mu(x,t)}\phi_n(x_
j)F(x,t).
\end{align*} 
Compared to the original proof in \cite{kenig_unique_2003}, the functions differ only by an overall sign from the corresponding ones in \cite{kenig_unique_2003}. For most of the argument this sign difference will not matter, since we work with the absolute value. However, the sign in front of $a^2$ in $\eqref{eq w}$ does play a role in the proof below. Note that it can be adapted by changing the role of $P_+$ and $P_-$ in the original proof.

Observe first that
\begin{equation}
\|\partial_{x_j}^lh\|_{L^\infty(\mathbb{R}^n \times [0,1])}\leq \frac{c_l}{n^{l+1}},
\end{equation}

\begin{equation} \|\partial^l_{x_j}a^2(x_j)\|_{L^\infty(\mathbb{R}^n)}\leq \frac{c_l}{n^l},
\end{equation}

\begin{equation}
\|\partial_{x_j}^lb(x_j)\|_{L^\infty(\mathbb{R}^n)}\leq \frac{c^l}{n^{l+1}}.
\end{equation} Now let us define 
$\eta\in C^\infty(\mathbb{R}^n)$ such that $0\leq \eta\leq 1$,
$$\eta(x) =\begin{cases}
  1 & |x|\leq 1/2 \\
  0 & |x|\geq 1,
\end{cases}
$$ 
and $$\chi_{\pm}(\xi_1) = \begin{cases}
    1 & \xi_1 >0 \ (\xi_1 <0) \\
    0 & \xi_1 <0 \ (\xi_1 >0).    
\end{cases}$$
Then we define, for $\epsilon \in (0,1)$, the Fourier multipliers
\begin{align*}
   \widehat{ P_\epsilon f} (\xi) &=\eta(\epsilon \xi)\hat{f}(\xi), \\ 
   \widehat{P_{\pm} f}(\xi) &= \chi_{\pm}(\xi_1)\hat{f}(\xi).
\end{align*} We derive the equation for $P_\epsilon P_+w_n,$
\begin{equation*}
i \partial_tP_\epsilon P_+ w_n +LP_\epsilon P_+w_n = P_\epsilon P_+ (Vw_n) + P_\epsilon P_+ (\tilde{F}_n) +   P_\epsilon P_+ ( hw_n) - P_\epsilon P_+ (a^2\partial_{x_j}w_n) +  P_\epsilon P_+ (itb\partial_{x_j}w_n).  
\end{equation*} Similarly, we get an equation for $\overline{P_\epsilon P_+ w_n},$ and it follows that
\begin{align*}
    \partial_t |P_\epsilon P_+ w_n|^2 + 2Im(L(P_\epsilon P_+ w_n) \overline{P_\epsilon P_+ w_n}) &= 2Im(P_\epsilon P_+(Vw_n)\overline{P_\epsilon P_+ w_n}) + 2Im(P_\epsilon P_+ (\tilde{F}_n)\overline{P_\epsilon P_+ w_n}) + \\& \ \ \ + 2Im(P_\epsilon P_+(hw_n)\overline{P_\epsilon P_+ w_n})  - 2Im(P_\epsilon P_+(a^2\partial_{x_j} w_n)\overline{P_\epsilon P_+ w_n}) \\ & \ \ \ +2tRe(P_\epsilon P_+(b\partial_{x_j}w_n)\overline{P_\epsilon P_+ w_n}). \end{align*} Since $w_n, \tilde{F}_n\in L^2$ for a.e. $t$
in $[0,1],$ we can integrate each term. 
Integration by parts yields
\begin{equation*}
    Im \int_{\mathbb{R}^n} L(P_\epsilon P_+ w_n) \overline{P_\epsilon P_+ w_n}) =0.
\end{equation*}
so that
\begin{align*}
    \partial_t \int_{\mathbb{R}^n} |P_\epsilon P_+ w_n|^2 dx & =\int_{\mathbb{R}^n} 2Im(P_\epsilon P_+(Vw_n)\overline{P_\epsilon P_+ w_n}) dx + \int_{\mathbb{R}^n}2Im(P_\epsilon P_+ (\tilde{F}_n)\overline{P_\epsilon P_+ w_n}) dx + \\& \ \ \ + \int_{\mathbb{R}^n}2Im(P_\epsilon P_+(hw_n)\overline{P_\epsilon P_+ w_n}) dx - \int_{\mathbb{R}^n}2Im(P_\epsilon P_+(a^2\partial_{x_j} w_n)\overline{P_\epsilon P_+ w_n}) dx \\ & \ \ \ +\int_{\mathbb{R}^n}2tRe(P_\epsilon P_+(b\partial_{x_j}w_n)\overline{P_\epsilon P_+ w_n}) dx \\ & = (1) + (2) + (3) +(4) + (5).
\end{align*}

First observe that
\begin{align*}
    |(1)| &\leq c \|V\|_{L^\infty(\mathbb{R}^n)}\|w_n\|_{L^2(\mathbb{R}^n)}^2, \\ |(2)| &\leq c\|F_n\|_{L^2(\mathbb{R}^n)}\|w_n\|_{L^2(\mathbb{R}^n)} \\ 
    |(3)| &\leq c \|h\|_{L^\infty(\mathbb{R}^n)}\|w_n\|_{L^2(\mathbb{R}^n)}^2 \leq \frac{c}{n} \|w_n\|_{L^2(\mathbb{R}^n)}^2.
\end{align*}
For the last two terms, we will use Calderón's commutator estimates \cite{calderon_commutators_1965}

\begin{align}\label{Calderon}
\|[P_{\pm},a]\partial_{x_j}f\|_{L^2(\mathbb{R}^n)}\leq c \|\partial_{x_j} a\|_{L^\infty(\mathbb{R}^n)}\|f\|_{L^2(\mathbb{R}^n)} 
\end{align}

\begin{align}\label{Calderon2}
\|\partial_{x_j}[P_{\pm},a]f\|_{L^2(\mathbb{R}^n)}\leq c \|\partial_{x_j} a\|_{L^\infty(\mathbb{R}^n)}\|f\|_{L^2(\mathbb{R}^n)}. 
\end{align} Since $P_\epsilon\in C^\infty_0$ is a nice Fourier multiplier, we also have the commutator estimates
\begin{align} \label{Calderon3}
    \|[P_\epsilon,a]\partial_{x_j}f\|_{L^2(\mathbb{R}^n)}\leq \frac cn \|f\|_{L^2(\mathbb{R}^n)}
\end{align}

\begin{align} \label{Calderon4}
    \|\partial_{x_j}[P_\epsilon,a]f\|_{L^2(\mathbb{R}^n)}\leq \frac cn \|f\|_{L^2(\mathbb{R}^n)}.
\end{align}

\begin{claim}
    \begin{align}
        -Im\int_{\mathbb{R}^n}P_\epsilon P_+(a^2(x_j)\partial_{x_j} w_n) \overline{P_\epsilon P_+ w_n}dx \leq O\left(\frac{\|w_n\|_{L^2(\mathbb{R}^n)}}{n}\right).
    \end{align}
\end{claim}

\begin{claim}
    \begin{align}
        2Re\int_{\mathbb{R}^n}P_\epsilon P_+ (b\partial_{x_j}w_n)\overline{P_\epsilon P_+w_n}
     = O\left(\frac{\|w_n\|_{L^2(\mathbb{R}^n)}^2}{n} \right)
     \end{align}
     uniformly in $\epsilon$ and $n$.
\end{claim}

Indeed, to prove the first claim, write
\begin{equation*}
    a^2(x_j) \partial_{x_j} w_n= a(x_j)\partial_{x_j}(a(x_j)w_n) - a(x_j)\partial_{x_j}a(x_j)w_n.
\end{equation*}

By the commutator estimates \eqref{Calderon} and \eqref{Calderon2}, it follows that
\begin{align*}
\int P_{+}(a^2(x_j)\partial_{x_j}w_n)\overline{P_+ w_n} dx &= \int P_+ a(x_j)\partial_{x_j}(a(x_j)w_n) \overline{P_+ w_n}dx + O\left(\frac{\|w_n\|_{L^2(\mathbb{R}^n)}}{n}\right) \\ & =\int a(x_j)P_+ \partial_{x_j}(a(x_j)w_n) \overline{P_+w_n} dx + O\left(\frac{\|w_n\|_{L^2(\mathbb{R}^n)}}{n}\right).
\end{align*}
Integrating by parts, and using that $a$ is real, and that $P_+$ and $\partial_{x_j}$ commutes, yields that

\begin{align*}
\int P_{+}(a^2(x_j)\partial_{x_j}w_n)\overline{P_+ w_n} dx &= -\int P_+(a(x_j)w_n) \overline{\partial_{x_j}(a P_+ w_n) }dx + O\left(\frac{\|w_n\|_{L^2(\mathbb{R}^n)}}{n}\right) \\ & = -\int P_+(a(x_j)w_n) \partial_{x_j} (\overline{P_+( a w_n)}) dx+ O\left(\frac{\|w_n\|_{L^2(\mathbb{R}^n)}}{n}\right) \\ &= \int \partial_{x_j} P_+(a w_n) \overline{P_+(a w_n)} dx + O\left(\frac{\|w_n\|_{L^2(\mathbb{R}^n)}}{n}\right).
\end{align*}

Similarly, also considering $P_\epsilon$, and using the commutator estimates \eqref{Calderon3} and \eqref{Calderon4}, we will get

\begin{align*}
\int P_\epsilon P_{+}(a^2(x_j)\partial_{x_j}w_n)\overline{P_\epsilon P_+ w_n} dx &=  \int \partial_{x_j}P_\epsilon  P_+(a w_n) \overline{P_\epsilon P_+(a w_n)} dx + O\left(\frac{\|w_n\|_{L^2(\mathbb{R}^n)}}{n}\right).
\end{align*}
Moreover,
\begin{align*}
    Im \int \partial_{x_j} P_\epsilon P_+(a w_n)\overline{P_\epsilon P_+(aw_n)}dx  & = Im \int_{\xi_j\geq 0}  (i\xi_j)\widehat{P_\epsilon P_+(aw_n)} \overline{\widehat{ P_\epsilon P_+(aw_n)}}dx \\ & = \int_{\xi_j \geq 0} \xi_j |P_\epsilon P_+(aw_n)|^2 dx \geq 0.
\end{align*}

Thus, 
\begin{align*}
  -Im\int_{\mathbb{R}^n}P_\epsilon P_+(a^2(x_j)\partial_{x_j} w_n) \overline{P_\epsilon P_+ w_n}dx \leq O\left(\frac{\|w_n\|_{L^2(\mathbb{R}^n)}}{n}\right).
\end{align*} The second claim follows similarly, but does not depend on the sign.

Hence, it follows that for some $c>0$ independent of both $n$ and $\epsilon,$ that

\begin{align}
    \partial_t \int_{\mathbb{R}^n}|P_\epsilon P_+ w_n|^2 \leq c\|V\|_{L^\infty(\mathbb{R}^n)}\|w_n\|_{L^2(\mathbb{R}^n)}^2 +c\|\tilde{F}\|_{L^2(\mathbb{R}^n)}\|w_n\|_{L^2(\mathbb{R}^n)}+\frac{c}{n}\|w_n\|_{L^2(\mathbb{R}^n)}^2.
\end{align}A similar argument for $P_-$ shows that
\begin{align}
    \partial_t \int_{\mathbb{R}^n}|P_\epsilon P_-w_n|^2 \geq -c\|V\|_{L^\infty(\mathbb{R}^n)}\|w_n\|_{L^2(\mathbb{R}^n)}^2 -c\|\tilde{F}\|_{L^2(\mathbb{R}^n)}\|w_n\|_{L^2(\mathbb{R}^n)}-\frac{c}{n}\|w_n\|_{L^2(\mathbb{R}^n)}^2.
\end{align}
Now, since 
 $$\sup_{t\in [0,1]} \|w_n(t)\|_{L^2(\mathbb{R}^n))}<\infty$$ for all $n\geq 0$, $\exists\, t_n\in [0,1]$ such that

 \begin{equation}
\|w_n(t_n)\|_{L^2(\mathbb{R}^n)}^2\geq 1/2 \| \sup_{t\in[0,1]}\|w_n(t)\|_{L^2(\mathbb{R}^n)}^2. 
 \end{equation}
 It follows now that
 \begin{align*}
      \frac12  \sup_{t\in[0,1]}\|w_n(t)\|_{L^2(\mathbb{R}^n)}^2 &\leq \|w_n(t_n)\|_{L^2(\mathbb{R}^n)}^2\\&=\|P_+ w_n(t_n)\|_{L^2(\mathbb{R}^n)}^2+\|P_-w_n(t_n)\|_{L^2(\mathbb{R}^n)}^2 \\ & = \lim_{\epsilon \to 0}\left(\|P_\epsilon P_+w_n(t_n)\|_{L^2(\mathbb{R}^n)}^2+\|P_\epsilon P_-w_n(t_n)\|_{L^2(\mathbb{R}^n)}^2\right) \\ & = \lim_{\epsilon \to 0}\Big(\int_{0}^{t_n} \partial_s \|P_\epsilon P_+w_n(s)\|_{L^2(\mathbb{R}^n)}^2 ds +\|P_\epsilon P_+ w_n(0)\|_{L^2(\mathbb{R}^n)}^2\\ & \ \ \ -\int_{t_n}^1\partial_s\|P_\epsilon P_-w_n(s)\|_{L^2(\mathbb{R}^n)}^2ds + \|P_\epsilon P_- w_n(1)\|_{L^2(\mathbb{R}^n)}^2\Big) \\ &\leq 2c \|V\|_{L^1([0,1],L^\infty(\mathbb{R}^n))} \sup_{t\in [0,1]} \|w_n(t)\|_{L^2(\mathbb{R}^n)}^2+2c\|\tilde{F}\|_{L^1([0,1],L^2(\mathbb{R}^n)}\sup_{t\in[0,1]}\|w_n(t)\|_{L^2(\mathbb{R}^n)}\\ & \ \ \ + \frac cn \sup_{t\in [0,1]}\|w_n(t)\|_{L^2(\mathbb{R}^n)}^2 + \|w_n(0)\|_{L^2(\mathbb{R}^n)}^2 + \|w_n(1)\|_{L^2(\mathbb{R}^n)}^2.
 \end{align*}
 By letting $n$ be large enough such that $\frac cn<1/4$ and choosing $\epsilon_0$ such that $2c\|V\|_{L^1([0,1],L^\infty(\mathbb{R}^n))}<\frac 18$, we deduce that for some $c>0,$

 \begin{align*}
      \sup_{t\in [0,1]}\|w_n(t)\|_{L^2(\mathbb{R}^n)}^2&\leq c\left(\|\tilde{F}\|_{L^1([0,1],L^2(\mathbb{R}^n)}\sup_{t\in[0,1]}\|w_n(t)\|_{L^2(\mathbb{R}^n)} + \|w_n(0)\|_{L^2(\mathbb{R}^n)}^2 + \|w_n(1)\|_{L^2(\mathbb{R}^n)}^2\right) \\ &\leq c\left(\|e^{\beta x_j}F\|_{L^1([0,1],L^2(\mathbb{R}^n))}^2  +\sup_{t\in[0,1]}\|w_n(t)\|_{L^2(\mathbb{R}^n)}^2 + \|e^{\beta x_j}u_0\|_{L^2(\mathbb{R}^n)}^2 + \|e^{\beta x_j}u(1)\|_{L^2(\mathbb{R}^n)}^2\right).
 \end{align*} By letting $n\to \infty$ the result follows.
\end{proof}
Now, we generalize the exponential decay in almost every direction.
\begin{lemma}\label{lambda direction}
There exists $\epsilon_0>0$ such that if $V:\mathbb{R}^n\times[0,1]\to \mathbb{C}$ satisfies \eqref{Vsmall} and $u\in C([0,1],L^2(\mathbb{R}^n))$ is a solution of \eqref{HSVF} with 
$F\in L^1([0,1],L^2(\mathbb{R}^n))$, the following holds. Let $\lambda \in \mathbb R^n$ satisfy $\langle \lambda, \lambda \rangle_{k,n-k} = \lambda^T J \lambda \neq 0$, where $J$ is the matrix defined in Subsection \ref{pseudo group}. Assume that $\|e^{\lambda \cdot x}u_0\|_{L^2(\mathbb{R}^n)}<\infty$, $\|e^{\lambda \cdot x}u(1)\|_{L^2(\mathbb{R}^n)}<\infty$ and $F\in L^1([0,1],L^2(e^{\lambda \cdot x} dx))$, then, for some $N>0$ independent of $\lambda$,
\begin{align} \label{lambda estimate}
    \sup_{t\in [0,1]}\|e^{\lambda \cdot x}u(t)\|_{L^2(\mathbb{R}^n)} \leq N(\|e^{\lambda \cdot x}u_0\|_{L^2(\mathbb{R}^n)} + \|e^{\lambda \cdot x}u(1)\|_{L^2(\mathbb{R}^n)} + \|F\|_{L^1([0,1],L^2(\mathbb{R}^n,e^{2\lambda \cdot x}dx))}) .
\end{align}
\end{lemma}
\begin{proof}
Let $O(k,n-k)$ be the pseudo/indefinite orthogonal group defined in Subsection \ref{pseudo group}.  Suppose first that $c:=\langle\lambda,\lambda\rangle_{k,n-k}=\lambda ^T J \lambda >0.$ As discussed in Subsection \ref{pseudo group}, it follows that there exist $B\in O(k,n-k)$ such that $(B^{-1})^T\lambda=\pm \sqrt{c}e_1$ \footnote{From Subsection \ref{pseudo group}, if $A\in O(k,n-k)$, then it is enough to choose $B=(A^T)^{-1}\in O(k,n-k).$}. By letting $y=Bx$, $v(y,t)=u(B^{-1}y,t)$, we have
\begin{align*}
    \int e^{2\lambda \cdot x}|u(x,t)|^2 dx &=\int e^{2\lambda \cdot B^{-1}y}|u(B^{-1}y,t)|^2 dy  = \int e^{2(B^{-1})^T\lambda \cdot y}|v(y,t)|^2 dy  = \int e^{\pm 2 \sqrt{c} y_1 }|v(y,t)|^2dy.
\end{align*}
Since $L$ is invariant under $O(k,n-k),$ we can apply Lemma \ref{general KPV lemma} to the solution $v$ with $\beta=\pm \sqrt{c}$, so that
\begin{align*}
    \int e^{2\lambda \cdot x}|u(x,t)|^2dx & \leq N(\|e^{\pm \sqrt{c} y_1}v(0)\|_{L^2(\mathbb{R}^n)} + \|e^{\pm \sqrt{c} y_1}v(1)\|_{L^2(\mathbb{R}^n)} + \|\tilde{F}\|_{L^1([0,1],L^2(\mathbb{R}^n,e^{2\pm \sqrt{c} y_1}dy))})^2,
\end{align*}
where $\tilde{F}(y,t) = F(B^{-1}y,t).$
By doing the reverse change of variable, $x=By,$ \eqref{lambda estimate} follows in this case.

If $c<0,$ there is $B\in O(k,n-k)$ such that $(B^{-1})^T\lambda = \pm \sqrt{-c} e_n$ and the argument follows by replacing $e_1$ with $e_n$. 
\end{proof}

Next we show how to obtain the Gaussian decay.

\begin{corollary}\label{gaussian decay 2}
If $u\in C([0,1],L^2(\mathbb{R}^n))$ is a solution to \eqref{HSEP}, $V$ is a bounded potential such that $$\lim_{R\to \infty}\|V\|_{L^1([0,1],L^\infty(\mathbb{R}^n\setminus B_R))}=0,$$ and for some $\gamma>0$ $\|e^{\gamma |x|^2}u_0\|_{L^2(\mathbb{R}^n)}+\|e^{\gamma |x|^2}u(1)\|_{L^2(\mathbb{R}^n)}<\infty,$ then $\exists \ N>0$ such that
\begin{align*}
    \sup_{t\in [0,1]}\|e^{\gamma |x|^2} u(t)\|_{L^2(\mathbb{R}^n)}\leq  N\left(\|e^{\gamma|x|^2}u_0\|_{L^2(\mathbb{R}^n)}+\|e^{\gamma |x|^2}u(1)\|_{L^2(\mathbb{R}^n)} + \sup_{t\in [0,1]}\|V\|_{L^\infty(\mathbb{R}^n)}\sup_{t\in[0,1]}\|u(t)\|_{L^2(\mathbb{R}^n)}\right).
\end{align*}
\end{corollary}
\begin{proof}
Let $R>0$ be large such that $$\|V\|_{L^1([0,1],L^\infty(\mathbb{R}^n\setminus B_R))}\leq \epsilon_0,$$ and define $V_R(x,t)=\mathbbm1_{\mathbb{R}^n\setminus B_R}V(x,t)$,  $F_R(x,t)=\mathbbm{1}_{B_R}V(x,t)u$.
Then $$\partial_t u = i(Lu + V_R(x,t)u + F_R(x,t)).$$ 
Let $\mathcal{C}=\{\lambda \in \mathbb{R}^n: \langle\lambda,\lambda\rangle_{k,n-k}=0$\}. Observe that since $e^{\gamma|x|^2}u_0$ and $e^{\gamma|x|^2}u(1)$ are in $L^2,$ it follows by Lemma \ref{lambda direction} that for $\lambda \in \mathbb{R}^n\setminus \mathcal{C}$ 
\begin{align*}
    \sup_{t\in [0,1]}\|e^{\lambda \cdot x}u(t)\|_{L^2(\mathbb{R}^n)} &\leq N\left(\|e^{\lambda \cdot x}u_0\|_{L^2(\mathbb{R}^n)} + \|e^{\lambda \cdot x}u(1)\|_{L^2(\mathbb{R}^n)} + \|F_{R}\|_{L^1([0,1],L^2(\mathbb{R}^n,e^{2\lambda \cdot x}dx))}\right) \\ &\leq N\left(\|e^{\lambda \cdot x}u_0\|_{L^2(\mathbb{R}^n)} + \|e^{\lambda \cdot x}u(1)\|_{L^2(\mathbb{R}^n)} + e^{|\lambda|R}\sup_{t\in[0,1]}\left(\|V(t)\|_{L^\infty(\mathbb{R}^n)}\|u(t)\|_{L^2(\mathbb{R}^n)}\right)\right).
\end{align*}
Now we replace $\lambda$ by $\lambda\sqrt{\gamma}$, square both sides and multiply by $e^{-|\lambda|^2/2}$. Since $\mathcal{C}$ has Lebesgue measure $0,$ it follows that
\begin{align*}
\int_{\mathbb{R}^n}\int_{\mathbb{R}^n}e^{2\sqrt{\gamma}\lambda \cdot x -|\lambda|^2/2} |u(x,t)|^2dx d\lambda&= \int_{\mathbb{R}^n\setminus \mathcal{C}}\int_{\mathbb{R}^n}e^{2\sqrt{\gamma}\lambda \cdot x -|\lambda|^2/2} |u(x,t)|^2dx d\lambda \\ &\leq N\Big(\int_{\mathbb{R}^n}\int_{\mathbb{R}^n} e^{2\sqrt{\gamma} \lambda \cdot x -|\lambda|^2/2}|u_0|^2 dx d\lambda + \int_{\mathbb{R}^n}\int_{\mathbb{R}^n} e^{2\sqrt{\gamma} \lambda \cdot x -|\lambda|^2/2}|u(1)|^2 dx d\lambda \\ \ \ &\quad  +  \int_{\mathbb{R}^n}\int_{\mathbb{R}^n} e^{2\sqrt{\gamma}|\lambda|R-|\lambda|/2}dx d\lambda\sup_{t\in[0,1]}\|V(t)\|_{L^\infty(\mathbb{R}^n)}^2\sup_{t\in[0,1]}\|u(t)\|_{L^2(\mathbb{R}^n)}^2 \Big).
\end{align*}
Using the identity 
$$\int_{\mathbb{R}^n} e^{2\sqrt{\gamma} \lambda \cdot x - |\lambda|^2/2}d\lambda = (2\pi)^{n/2}e^{2\gamma |x|^2},$$ we deduce that for some $N>0$

\begin{align*}
\|e^{\gamma|x|^2}u(t)\|_{L^2(\mathbb{R}^n)}^2 &\leq N \Big(\|e^{\gamma |x|^2}u_0\|_{L^2(\mathbb{R}^n)}^2 + \|e^{\gamma |x|^2}u(1)\|_{L^2(\mathbb{R}^n)}^2 + \sup_{t\in [0,1]}\|V\|_{L^\infty(\mathbb{R}^n)}\sup_{t\in [0,1]}\|u(t)\|_{L^2(\mathbb{R}^n)}^2\Big).
\end{align*}
\end{proof}

\subsection{Conclusion of the proof of Theorem \ref{theorem 3}}
\begin{proof}[Proof of Theorem \ref{theorem 3}]
Let $u$ be a solution to the hyperbolic Schrödinger equation,
\begin{equation*}
    \partial_t u = i(Lu + Vu), 
\end{equation*}
with $u(x,0)=u_0,$ and such that either \eqref{cond1} or \eqref{cond2} are satisfied. We will need to distinguish between the two conditions on the potential. Observe first that if $V$ satisfies \eqref{cond2}, then Corollary \ref{gaussian decay 2} implies \eqref{IG1}. To prove \eqref{IG1} when $V$ satisfies \eqref{cond1} and to prove \eqref{IG2} in both cases, we will do a parabolic regularization argument, and apply Lemma \ref{Lemma3} and Lemma \ref{Lemma 4 paper}. 

Using the Duhamel formula, we can write
\begin{equation}
    u(t) = e^{itL}u_0 + i\int_{0}^te^{iL(t-s)}(Vu)(s)ds.
\end{equation}
Let $\epsilon\in (0,1)$ and $u_\epsilon$ be the solution to \eqref{PAR} with $A=\epsilon$ and $B=1$, such that
\begin{equation}
    \begin{cases}
    \partial_t u_{\epsilon}  = \epsilon\Delta u_\epsilon + i(Lu_\epsilon +  F_\epsilon)
    \\
    u_\epsilon(0) = u_0,
    \end{cases}
\end{equation}
where $F_\epsilon(t) = e^{\epsilon \Delta t}(Vu(t)).$

\begin{claim}
For $\epsilon>0$ the operator $\epsilon \Delta + iL$ generates a $C_0$ semigroup and \begin{equation} \label{semigroup property}
S(t)=e^{(\epsilon \Delta + iL)t} =e^{\epsilon\Delta t+iLt}=e^{\epsilon\Delta t}e^{iLt}.\end{equation}
\end {claim}The claim can be verified by a direct computation, using that 
$e^{\epsilon t\Delta}u_0=\left(e^{-\epsilon t |\xi|^2}\hat{u}_0\right)\widecheck{}$ and that $e^{itL}u_0 = \left(e^{-it(|\xi_+|^2-|\xi_{-}|^2)}\hat{u}_0\right)\widecheck{}$.
Thus,
\begin{align*}
u_\epsilon(t) = e^{(\epsilon \Delta +iL)t}u_0 + i \int_{0}^te^{(\epsilon \Delta + iL)(t-s)}F_\epsilon(s) ds 
= e^{\epsilon \Delta t}u(t) . 
\end{align*}
Moreover,
$u_\epsilon(t) \to u(t)$ in $L^2(\mathbb{R}^n)$ when $\epsilon\to 0.$ 
We start by verifying that the conditions in Lemma \ref{Lemma3} are satisfied for $u_\epsilon$ and $F_\epsilon$. Notice first that since $\epsilon>0$, we have  $u_\epsilon \in C([0,1],L^2(\mathbb{R}^n)) \cap L^2([0,1],H^1(\mathbb{R}^n)).$ Let $v = e^{\epsilon t\Delta}u(1)$. Then 
\begin{align*}
    \begin{cases}
    \partial_t v=\epsilon \Delta v \\ v(0) = u(1),
    \end{cases}
\end{align*}
so by Lemma \ref{lemma 1 b)} with $A=\epsilon, B=0$ and $T=1$,
\begin{equation}\label{u_epsilon(1)}
\|e^{\gamma_\epsilon |x|^2}u_\epsilon(1)\|_{L^2(\mathbb{R}^n)}\leq \|e^{\gamma|x|^2}u(1)\|_{L^2(\mathbb{R}^n)},
\end{equation}
where $\gamma_\epsilon  = \frac{\gamma}{1+8\gamma\epsilon}.$
Since $u_\epsilon(0)=u_0,$
\begin{equation}\label{u_epsilon(0)}
\|e^{\gamma_\epsilon |x|^2}u_\epsilon (0)\|_{L^2(\mathbb{R}^n)} \leq\|e^{\gamma |x|^2}u_\epsilon (0) \|_{L^2(\mathbb{R}^n)} =\|e^{\gamma|x|^2}u_0\|_{L^2(\mathbb{R}^n)}.
\end{equation}
In order to prove that 
\begin{equation*}
M_{2,\epsilon} = \sup_{t\in[0,1]}\frac{\|e^{\gamma_\epsilon|x|^2}F_\epsilon (t)\|_{L^2(\mathbb{R}^n)}}{\|u_\epsilon(t)\|_{L^2(\mathbb{R}^n)}} <\infty
\end{equation*}
we start by observing that for all $t\in [0,1]$
\begin{equation} \label{bound F_epsilon}
   \|e^{\gamma_\epsilon |x|^2}F_{\epsilon}(t)\|_{L^2(\mathbb{R}^n)}\leq \|e^{\frac{\gamma |x|^2}{1+8\gamma \epsilon t}}F_{\epsilon}(t)\|_{L^2(\mathbb{R}^n)}\leq  \|e^{\gamma |x|^2}(Vu)(t)\|_{L^2(\mathbb{R}^n)},
\end{equation}
which follows by a similar application of Lemma \ref{lemma 1 b)} to the function $G(s)=e^{\epsilon s \Delta}(Vu)(t),$ for $t\in [0,1]$ fixed.  
When $V$ satisfies \eqref{cond1} we observe that 
\begin{equation*}
    M_{2,\epsilon} \leq \sup_{t\in[0,1]}\frac{\|e^{\gamma|x|^2}V(t)\|_{L^\infty(\mathbb{R}^n)}\|u(t)\|_{L^2(\mathbb{R}^n)}}{\|u_\epsilon(t)\|_{L^2(\mathbb{R}^n)}},
\end{equation*}
so that the numerator will be finite. Thus, to prove $M_{2,\epsilon}<\infty,$ so that we can apply Lemma \ref{Lemma3}, we are left to bound $\|u_\epsilon(t)\|_{L^2(\mathbb{R}^n)}$ from below. The energy method shows
\begin{equation}
\frac{d}{dt}\|u_\epsilon(t)\|_{L^2(\mathbb{R}^n)}^2\leq2\|F_\epsilon (t)\|_{L^2(\mathbb{R}^n)}\|u_\epsilon(t)\|_{L^2(\mathbb{R}^n)},
\end{equation}
and since
$e^{\epsilon \Delta t}$ is a $C_0$ semigroup
\begin{equation}
\|u_\epsilon(t)\|_{L^2(\mathbb{R}^n)}\leq \|u(t)\|_{L^2(\mathbb{R}^n)}
\end{equation}
\begin{equation}\label{F bound}
\|F_\epsilon(t)\|_{L^2(\mathbb{R}^n)} \leq \|Vu\|_{L^2(\mathbb{R}^n)} \leq\|V\|_{L^\infty(\mathbb{R}^n)}\|u(t)\|_{L^2(\mathbb{R}^n)}.
\end{equation} Thus, by also using Lemma \ref{Energyestimate} it follows that
\begin{equation}
    \frac{d}{dt}\|u_\epsilon(t)\|_{L^2(\mathbb{R}^n)}\leq N_V\|V\|_{L^\infty(\mathbb{R}^n)}\|u_0\|_{L^2(\mathbb{R}^n)},
\end{equation}
where $N_V=e^{\sup_{t\in [0,1]}\|Im V(t)\|_{L^\infty(\mathbb{R}^n)}}.$
Let $0=t_0,t_1,\dots,t_k=1$, be a uniform partition of $[0,1].$ Fix $t_{i-1}\leq t\leq t_i,$ for $0\leq i\leq k$, and integrate from $t$ to $t_i$ to deduce that 
\begin{align}\label{integrated}
    \|u_\epsilon(t_i)\|_{L^2(\mathbb{R}^n)}\leq \|u_\epsilon(t)\|_{L^2(\mathbb{R}^n)} + N_V\sup_{t\in [0,1]}\|V\|_{L^\infty(\mathbb{R}^n)}\|u_0\|_{L^2(\mathbb{R}^n)}(t_{i}-t_{i-1}).
\end{align}
Now since $\|u_\epsilon(t)\|_{L^2(\mathbb{R}^n)}\to \|u(t)\|_{L^2(\mathbb{R}^n)}$ when $\epsilon\to 0$, $\exists \ \epsilon_0>0$ such that for $0<\epsilon<\epsilon_0$ 
\begin{equation}
    \|u_\epsilon(t_i)\|\geq\frac{1}{2}\|u(t_i)\|_{L^2(\mathbb{R}^n)}\geq \frac{1}{2N_V}\|u_0\|_{L^2(\mathbb{R}^n)}.
\end{equation} Thus, if we choose $k$ such that $N_V\sup_{t\in [0,1]}\|V\|_{L^\infty(\mathbb{R}^n)}(t_i-t_{i-1})\leq \frac{1}{4N_V},$ \eqref{integrated} implies that for all $t\in [0,1]$
\begin{align}
\|u_\epsilon(t)\|_{L^2(\mathbb{R}^n)}\geq\|u_\epsilon(t_i)\|_{L^2(\mathbb{R}^n)} - \frac{1}{4N_V}\|u_0\|_{L^2(\mathbb{R}^n)} \geq \frac{1}{4N_V}\|u_0\|_{L^2(\mathbb{R}^n)},
\end{align}
and this bound is independent of $t$. Thus, if $V$ satisfies \eqref{cond1},
\begin{equation}\label{M21}
    M_{2,\epsilon}\leq 4  \sup_{t\in [0,1]} \|e^{\gamma |x|^2}V(t)\|_{L^\infty(\mathbb{R}^n)}N_V^2:=C_1(V),
\end{equation}

By Lemma \ref{Lemma3}, \eqref{u_epsilon(0)} and \eqref{u_epsilon(1)}, it follows that $\|e^{\gamma_\epsilon |x|^2}u_\epsilon(t)\|_{L^2(\mathbb{R}^n)}$ is logarithmically convex in $[0,1]$ and that for all $t\in [0,1]$
\begin{align}\label{1st}
\nonumber\|e^{\gamma_\epsilon|x|^2}u_\epsilon(t)\|_{L^2(\mathbb{R}^n)}&\leq e^{N(M_{2_\epsilon}^2 + M_{2_\epsilon})}\|e^{\gamma_\epsilon |x|^2}u_\epsilon(0)\|_{L^2(\mathbb{R}^n)}^{1-t}\|e^{\gamma_\epsilon |x|^2}u_\epsilon(1)\|_{L^2(\mathbb{R}^n)}^t \\ &\leq e^{N(C_1(V)^2 +C_1(V))}\|e^{\gamma |x|^2}u_0\|_{L^2(\mathbb{R}^n)}^{1-t} \|e^{\gamma|x|^2}u(1)\|_{L^2(\mathbb{R}^n)}^t 
\end{align}
for some $N>0.$ 

If $V$ satisfies \eqref{cond2}, we also have the bound, again using Lemma \ref{lemma 1 b)},
\begin{equation} \label{sup 2}
    \sup_{t\in[0,1]}\|e^{\gamma_{\epsilon}|x|^2}u_\epsilon (t)\|_{L^2(\mathbb{R}^n)}\leq \sup_{t\in[0,1]}\|e^{\gamma |x|^2}u(t)\|_{L^2(\mathbb{R}^n)},
\end{equation}
and by \eqref{bound F_epsilon} that
\begin{equation} \label{sup 2F}
    \sup_{t\in[0,1]}\|e^{\gamma_{\epsilon}|x|^2}F_\epsilon (t)\|_{L^2(\mathbb{R}^n)}\leq \sup_{t\in [0,1]}\|V(t)\|_{L^\infty(\mathbb{R}^n)}\|e^{\gamma |x|^2}u(t)\|_{L^2(\mathbb{R}^n)}.
\end{equation}
Then by Lemma \ref{Lemma 4 paper}, both when $V$ satisfies \eqref{cond1}, and when $V$ satisfies \eqref{cond2}, it follows that for all $0<t<1$
\begin{align}\label{2nd}
   \nonumber &\|\sqrt{t(1-t)} e^{\gamma_\epsilon|x|^2}\nabla u_\epsilon\|_{L^2(\mathbb{R}^n \times [0,1])} +  \|\sqrt{t(1-t)}|x| e^{\gamma_\epsilon|x|^2}u_\epsilon\|_{L^2(\mathbb{R}^n \times [0,1])} \\ &\leq  N_1(\epsilon, \gamma)\left(\sup_{t\in [0,1]}\|e^{\gamma_\epsilon |x|^2}u_\epsilon(t)\|_{L^2(\mathbb{R}^n)}+\sup_{t\in[0,1]}\|e^{\gamma_{\epsilon}|x|^2}F_\epsilon\|_{L^2(\mathbb{R}^n)}\right) 
\end{align}
where $N_1=N_1(\epsilon,\gamma)$ is the constant in Lemma \ref{Lemma 4 paper}\footnote{Remark that $N_1$ remains bounded when we let $\epsilon \to 0,$ since $\epsilon^2+1\geq 1.$  }. If $V$ satisfies \eqref{cond1}, then
\begin{align*}
   \nonumber &\|\sqrt{t(1-t)} e^{\gamma_\epsilon|x|^2}\nabla u_\epsilon\|_{L^2(\mathbb{R}^n \times [0,1])} +  \|\sqrt{t(1-t)}|x| e^{\gamma_\epsilon|x|^2}u_\epsilon\|_{L^2(\mathbb{R}^n \times [0,1])} \\ &\leq NN_1e^{C_1(V)^2 + C_1(V)}\Big(\|e^{\gamma |x|^2}u_0\|_{L^2(\mathbb{R}^n)} + \|e^{\gamma |x|^2}u(1)\|_{L^2(\mathbb{R}^n)} + C_1(V) \sup_t\|u(t)\|_{L^2(\mathbb{R}^n)}\Big),
\end{align*}

and if $V$ satisfies \eqref{cond2}, then, by also using Corollary \ref{gaussian decay 2},
\begin{align*}
&\|\sqrt{t(1-t)} e^{\gamma_\epsilon|x|^2}\nabla u_\epsilon\|_{L^2(\mathbb{R}^n \times [0,1])} +  \|\sqrt{t(1-t)}|x| e^{\gamma_\epsilon|x|^2}u_\epsilon\|_{L^2(\mathbb{R}^n \times [0,1])} \\ &\leq N N_1(\sup_{t\in [0,1]}\|V\|_{L^\infty(\mathbb{R}^n)})^2\left( \|e^{\gamma |x|^2}u_0\|_{L^2(\mathbb{R}^n)}+\|e^{\gamma |x|^2}u(1)\|_{L^2(\mathbb{R}^n)} +\sup_{t\in [0,1]}\|u(t)\|_{L^2(\mathbb{R}^n)} \right).
\end{align*}

Finally, letting $\epsilon$ to 0, we deduce in the case $V$ satisfies \eqref{cond1}
\begin{equation*}
    \|e^{\gamma |x|^2}u(t)\|_{L^2(\mathbb{R}^n)}\leq e^{N(C_1(V) + C_1(V)^2)}\|e^{\gamma |x|^2}u_0\|_{L^2(\mathbb{R}^n)}^{1-t}\|e^{\gamma |x|^2}u(1)\|_{L^2(\mathbb{R}^n)}^t,
\end{equation*}
and
\begin{align*}
&\|\sqrt{t(1-t)} e^{\gamma_\epsilon|x|^2}\nabla u_\epsilon\|_{L^2(\mathbb{R}^n \times [0,1])} +  \|\sqrt{t(1-t)}|x| e^{\gamma_\epsilon|x|^2}u_\epsilon\|_{L^2(\mathbb{R}^n \times [0,1])} \\ &\leq Ne^{N (C_1(V) + C_1(V)^2)}\left( \|e^{\gamma|x|^2}u_0\|_{L^2(\mathbb{R}^n)} + \|e^{\gamma |x|^2}u(1)\|_{L^2(\mathbb{R}^n)} + C_1(V) \sup_{t\in [0,1]} \|u(t)\|_{L^2(\mathbb{R}^n)}\right),
\end{align*} and if $V$ satisfies \eqref{cond2}, from Corollary \ref{gaussian decay 2},
\begin{align*}
\|e^{\gamma |x|^2}u(t)\|_{L^2(\mathbb{R}^n)} \leq N\left(\|e^{\gamma |x|^2} u_0\|_{L^2(\mathbb{R}^n)} + \|e^{\gamma |x|^2}u(1)\|_{L^2(\mathbb{R}^n)}+\sup_{t\in [0,1]}\|V\|_{L^\infty(\mathbb{R}^n)}\sup_{t\in [0,1]}\|u(t)\|_{L^2(\mathbb{R}^n)}\right),
\end{align*}
and
\begin{align*}
\nonumber &\|\sqrt{t(1-t)} e^{\gamma|x|^2}\nabla u\|_{L^2(\mathbb{R}^n \times [0,1])} +  \|\sqrt{t(1-t)}|x| e^{\gamma|x|^2}u\|_{L^2(\mathbb{R}^n \times [0,1])} \\ &\leq N(\sup_{t\in [0,1]}\|V\|_{L^\infty(\mathbb{R}^n)})^2\left( \|e^{\gamma |x|^2}u_0\|_{L^2(\mathbb{R}^n)}+\|e^{\gamma |x|^2}u(1)\|_{L^2(\mathbb{R}^n)} +\sup_{t\in [0,1]}\|u(t)\|_{L^2(\mathbb{R}^n)} \right).
\end{align*}

\end{proof}

\section{proof of the main result, Theorem \ref{hyperbolic result}}
We first prove the following Carleman estimate in both space and time, for compactly supported functions. The proof follows exactly as for the elliptic case in \cite{escauriaza_hardys_2008}.
\begin{lemma}\label{Carleman estimate in n+1}
For $\phi(x,t)=\mu|x+Rt(1-t)\tilde{\xi}|^2-(1+\epsilon)R^2t(1-t)/16\mu$, $\mu>0,$ $\epsilon>0,$ $R>0,$ for some $\xi \in \mathbb{R}^n$, $|\xi|=1$ and $g\in C^\infty_0(\mathbb{R}^{n+1}),$
\begin{equation}
    \label{Carleman}
    R\sqrt{\frac{\epsilon}{8\mu}}\|e^{\phi(x,t)}g\|_{L^2(\mathbb{R}^{n+1})} \leq \|e^{\phi(x,t)}(\partial_t - iL)g\|_{L^2(\mathbb{R}^{n+1})}.
\end{equation}
\end{lemma}

\begin{proof}
 Let $f=e^{\phi}g.$ Then
\begin{align*}e^{\phi}(\partial_t -iL)e^{-\phi}f  &=-\partial_t \phi f + \partial_t f -i(\nabla \phi \cdot \nabla_H \phi - 2i \nabla \phi \cdot \nabla_H - L\phi + L)f \\ & = \partial_t f-\mathcal{S}f-\mathcal{A}f, \end{align*}where $\mathcal{S}$ is the symmetric operator
$$\mathcal{S} = \partial_t \phi - i(2\nabla \phi \cdot \nabla_H + L\phi)$$
and $\mathcal{A}$ is the skew symmetric operator
$$\mathcal{A} = i(\nabla \phi \cdot \nabla_H \phi + L). $$ From \eqref{commutator}  with $A=0$ and $B=1$

\begin{equation}
    \partial_t \mathcal{S} + [\mathcal{S,A}] = \partial_t^2\phi -4i \nabla(\partial_t \phi) \cdot \nabla_H - 2iL(\partial_t \phi) + 4\nabla \phi \cdot D^2_H\phi \nabla \phi - 4\nabla \cdot D^2_H \phi \nabla  -L^2 \phi
\end{equation}
and for $\phi(x,t)=\mu|x+Rt(1-t)\tilde{\xi}|^2-(1+\epsilon)R^2t(1-t)/16\mu$,
\begin{align*}
    \partial_t \mathcal{S} + [\mathcal{S,A}]  = & 2\mu R^2(1-2t)^2|\xi|^2 -4\mu R \tilde{\xi}(x+Rt(1-t)\tilde{\xi}) + \frac{(1+\epsilon)R^2}{8\mu} + 32\mu^3|x+Rt(1-t)\tilde{\xi}|^2 \\  &- 8\mu \Delta - 8i\mu R(1-2t) \xi \cdot \nabla. 
\end{align*}
Thus, 
\begin{align*}
    \langle \partial_t \mathcal{S}f + [\mathcal{S,A}]f,  \textit{}f\rangle_{L^2(\mathbb{R}^n)} =& \int_{\mathbb{R}^n}2\mu R^2(1-2t)^2 |f|^2 dx -\int_{\mathbb{R}^n}4\mu R \tilde{\xi}(x+Rt(1-t)\tilde{\xi})|f|^2  dx   \\ &+\frac{(1+\epsilon)R^2}{8\mu}\int_{\mathbb{R}^n}|f|^2 dx  + 32 \mu^3\int_{\mathbb{R}^n}|x+Rt(1-t)\tilde{\xi}|^2 |f|^2 + 8\mu\int_{\mathbb{R}^n} |\nabla f|^2 dx \\ & -8i\mu R\int_{\mathbb{R}^n}(1-2t)\xi \cdot \nabla f \bar{f}dx \\ & = (1)+(2)+(3)+(4)+(5)+(6). 
\end{align*}
By Cauchy Schwartz and Young's inequalities, and using that $|\xi|=1,$ 
we can show that 

\begin{align*}
   (4) - |(2)| \geq -\frac{R^2}{8\mu} \|f\|^2_{L^2(\mathbb{R}^n)},
\end{align*}
and that

\begin{align*}
|(6)| \leq 2\mu R^2 \|(1-2t)f\|_{L^2(\mathbb{R}^n)}^2 +8\mu  \|\nabla f\|_{L^2(\mathbb{R}^n)}^2,
\end{align*}
so that 
\begin{align*}
 \langle \partial_t \mathcal{S}f + [\mathcal{S,A}]f,  \textit{}f\rangle_{L^2(\mathbb{R}^n)}\geq \frac{\epsilon R^2}{8\mu}\|f\|_{L^2(\mathbb{R}^n)}^2.
\end{align*}

Finally, since
\begin{equation}
   \|\partial_t f - \mathcal{S}f - \mathcal{A}f\|_{L^2(\mathbb{R}^{n+1})}^2   = \int_{\mathbb{R}} \langle \partial_t f - \mathcal{S}f -\mathcal{A}f,  \partial_t f - \mathcal{S}f -\mathcal{A}f \rangle_{L^2(\mathbb{R}^n)} \geq \int_{\mathbb{R}} \langle \partial_t \mathcal{S} + [\mathcal{S,A}]f,f\rangle_{L^2(\mathbb{R}^n)} ,
\end{equation}
it follows that 
\begin{align*}
\frac{\epsilon R^2}{8\mu}\|e^{\phi} g\|_{L^2(\mathbb{R}^{n+1})}^2 \leq \|e^{\phi}(\partial_t - iL)g\|_{L^2(\mathbb{R}^{n+1})}^2,
\end{align*} which proves the lemma.

\end{proof}

Now we can finally prove the main result, Theorem \ref{hyperbolic result}. By Theorem \ref{theorem 3} we are now able to justify all computations, and the proof follows as in \cite{escauriaza_hardys_2008}.
\begin{proof}[Proof of Theorem \ref{hyperbolic result}]

Let $u$ be described as in the hypothesis in Theorem \ref{hyperbolic result}. Let $\tilde{u}, \tilde{V}$ be defined through the Appell transformation in Lemma \ref{conformalappelltransformation}. Then
$$\partial_t \tilde{u}=i(L \tilde{u} + \tilde{V}(x,t)\tilde{u}) \ \ \ \text{ in } \ \mathbb{R}^n\times [0,1],$$ and for $\gamma=\frac{1}{\alpha \beta},$ $\gamma>\frac{1}{2},$  $\|e^{\gamma |x|^2}\tilde{u}(0)\|_{L^2(\mathbb{R}^n)}, \|e^{\gamma |x|^2}\tilde{u}(1)\|_{L^2(\mathbb{R}^n)}$ are both finite.
Let $R>0,$ and let $\mu$ and $\epsilon>0$ small enough satisfy 
\begin{equation}\label{fix epsilon and mu}
    \frac{(1+\epsilon)^{3/2}}{2(1-\epsilon)^3} < \mu \leq \frac{\gamma}{1+\epsilon},
\end{equation}
and 
\begin{equation}
    \frac{1-\epsilon}{2} \geq \frac {1}{R}.
\end{equation}
\begin{remark} $\frac{(1+\epsilon)^{5/2}}{(1-\epsilon)^{3}}$ will be close to $1$ if $\epsilon$ is small enough, and since $\gamma>\frac{1}{2}$, there exists $\mu$ such that (\ref{fix epsilon and mu}) is satisfied.
\end{remark}

To be able to use the Carleman estimate in the previous lemma, we need to define a function $g$ with compact support in both space and time.
Let $\theta\in C^\infty_0(\mathbb{R}^n)$ be such that 
\begin{equation*}
    \theta(x) = \begin{cases}
        1, & |x|\leq 1,\\
        0, & |x|>2,
    \end{cases}
\end{equation*}
and for $M\geq R,$ $\theta_M(x) = \theta(\frac{x}{M}),$
Then we define $\eta_R\in C^\infty_0 ([0,1])$
\begin{equation}
   \eta_R(t)= \begin{cases}
        1 & t \in [\frac{1}{R}, 1-\frac{1}{R}],\\
        0 & t\in [0,\frac{1}{2R}] \cup [1-\frac{1}{2R},1],
    \end{cases}
\end{equation}
and it follows
\begin{align*}
\|\nabla \theta_M(x)\|_{L^\infty(\mathbb{R}^n)} & \leq \frac{N}{M}, \quad 
\|L \theta_M(x)\|_{L^\infty(\mathbb{R}^n)} \leq \frac{N}{M^2}, \quad  
\|\eta_R'(t)\|_{L^\infty([0,1])}  \leq N R
\end{align*} for some constant $N>0.$ We define $g(x,t)=\tilde{u}(x,t)\theta_M(x)\eta_R(t).$ A direct computation shows that 
\begin{equation}\label{gequation}
    \partial_t g - i(L g + \tilde{V}g) =  \tilde{u}\theta_M \eta_R' -i\eta_R(\tilde{u}L\theta_M + 2\nabla \theta_M \cdot \nabla_H \tilde{u}).
\end{equation}
Observe that for the first term on the right-hand side of (\ref{gequation})
$$\text{supp}(\tilde{u}\theta_M \eta'_R ) \subset \{ (x,t): |x|<2M, \   t\in [\frac{1}{2R},\frac{1}{R}] \cup [1-\frac{1}{R}, 1-\frac{1}{2R}]\},$$ and on this region we have, using Young's inequality and (\ref{fix epsilon and mu}),
\begin{align} \label{support1}
\mu|x+Rt(1-t)\tilde{\xi}|^2 \leq \gamma|x|^2 +\frac{\gamma}{\epsilon}.
\end{align}
For the second term on the right-hand side of (\ref{gequation}) we have
$$\text{supp}((2\nabla \theta_M \cdot \nabla \tilde{u} + \tilde{u}\Delta \theta_M) \eta_R)  \subset \{(x,t): M\leq |x| \leq 2M, \ t\in (\frac{1}{2R}, 1-\frac{1}{2R})\},$$  so that
\begin{align} \label{support2}
    \mu|x+Rt(1-t)\tilde{\xi}|^2 & \leq \gamma |x|^2 + \gamma \frac{R^2}{\epsilon}.
\end{align}
Since $g$ has compact support in $\mathbb{R}^n \times [0,1]$ we apply the Carleman estimate in Lemma \ref{Carleman estimate in n+1}. Let $$\phi(x,t) = \mu|x+Rt(1-t)\tilde{\xi}|^2 - \frac{(1+\epsilon)R^2 t(1-t)}{16 \mu}.$$ Combining the Carleman estimate, the bounds for $\nabla \theta_M, L \theta_M$ and $\eta_R'$ together with (\ref{support1}) and (\ref{support2}) we get
\begin{align}
    R\|e^\phi g\|_{L^2(\mathbb{R}^n \times [0,1])} & \leq N(\epsilon,\mu)(\|\tilde{V}\|_{L^\infty(\mathbb{R}^n\times [0,1])}\|e^{\phi}g\|_{L^2(\mathbb{R}^n \times [0,1])} \nonumber  + R e^{\gamma/\epsilon} \sup_{t\in [0,1]} \|e^{\gamma |x|^2} \tilde{u} \|_{L^2(\mathbb{R}^n)} \nonumber \\ &  \ \  + \frac{1}{M} e^{\gamma R^2/\epsilon} \|e^{\gamma |x|^2} (|\tilde{u}| + |\nabla \tilde{u}|)\|_{L^2(\mathbb{R}^n \times [\frac{1}{2R}, 1-\frac{1}{2R}])}).
\end{align}
For $R\geq 2N(\epsilon, \mu)\|\tilde{V}\|_{L^\infty(\mathbb{R}^n \times [0,1])},$ 
\begin{align} \label{Est. after using lemma7}
    R\|e^\phi g\|_{L^2(\mathbb{R}^n \times [0,1])} & \leq N(\epsilon,\mu,\gamma) \left(R  \sup_{t\in [0,1]} \|e^{\gamma |x|^2} \tilde{u} \|_{L^2(\mathbb{R}^n)}  + \frac{1}{M} e^{\gamma R^2/\epsilon} \|e^{\gamma |x|^2} (|\tilde{u}| + |\nabla \tilde{u}|)\|_{L^2(\mathbb{R}^n \times [\frac{1}{2R}, 1-\frac{1}{2R}])}\right).
\end{align}

By Theorem \ref{theorem 3} we have that in both cases of the potential $V$
\begin{equation*} \|e^{\gamma |x|^2} (|\tilde{u}| + |\nabla \tilde{u}|)\|_{L^2(\mathbb{R}^n \times [\frac{1}{2R}, 1-\frac{1}{2R}])} <\infty,\end{equation*}
so by letting M to $+\infty$, the last term on the right-hand side of (\ref{Est. after using lemma7}) goes to zero, and we obtain that
\begin{equation}
    \|e^\phi g\|_{L^2(\mathbb{R}^n \times [0,1])} \leq N(\epsilon,\gamma,\mu) \sup_{t\in [0,1]} \|e^{\gamma |x|^2} \tilde{u} \|_{L^2(\mathbb{R}^n)}.
\end{equation}
In $B_{\epsilon (1-\epsilon)^2 \frac{R}{4}} \times [\frac{1-\epsilon}{2},\frac{1+\epsilon}{2}]$, we have $|x|<\epsilon (1-\epsilon)^2 \frac{R}{4}<R\leq M$ and $t\in [\frac{1-\epsilon}{2},\frac{1+\epsilon}{2}] \subset [\frac{1}{R}, 1-\frac{1}{R}]$, which implies that in $B_{\epsilon (1-\epsilon)^2 \frac{R}{4}} \times [\frac{1-\epsilon}{2},\frac{1+\epsilon}{2}], g=\tilde{u}.$
Moreover,  \eqref{fix epsilon and mu} yields
\begin{align} \label{estimate on phi in the ball}
    \phi(x,t) & \geq \mu(Rt(1-t) - |x|)^2 -\frac{(1+\epsilon)R^2t(1-t)}{16\mu} \nonumber\\ & \geq  \frac{R^2}{64}(4\mu^2(1-\epsilon)^6 - (1+\epsilon)^3) > 0,
\end{align}
so that
\begin{align*}  
R \|e^{\frac{R^2}{64}(4\mu^2(1-\epsilon)^6 - (1+\epsilon)^3)} g\|_{L^2(B_{\epsilon (1-\epsilon)^2 \frac{R}{4}} \times [\frac{1-\epsilon}{2},\frac{1+\epsilon}{2}])} &\leq N(\epsilon,\gamma, \mu)\sup_{t\in [0,1]} \|e^{\gamma |x|^2} \tilde{u} \|_{L^2(\mathbb{R}^n)},
\end{align*}
or equivalently, since $g=\tilde{u}$ in $B_{\epsilon (1-\epsilon)^2 \frac{R}{4}} \times [\frac{1-\epsilon}{2},\frac{1+\epsilon}{2}]$,
\begin{equation}\label{bound on u in the ball}
e^{N_{\epsilon,\mu} R^2} \|\tilde{u}\|_{L^2(B_{\epsilon (1-\epsilon)^2 \frac{R}{4}} \times [\frac{(1-\epsilon)}{2},\frac{(1+\epsilon)}{2}])}\leq N(\gamma,\epsilon,\mu)\sup_{t\in [0,1]} \|e^{\gamma |x|^2} \tilde{u} \|_{L^2(\mathbb{R}^n)}.
\end{equation}
We also have that
\begin{align}
   \nonumber  \|\tilde{u}(t)\|^2_{L^2(\mathbb{R}^n)} &\leq \int_{B_{\epsilon(1-\epsilon)^2\frac{R}{4}}} |\tilde{u}(t)|^2 dx + e^{-2\gamma \epsilon^2(1-\epsilon)^4 \frac{R^2}{16}}\int_{|x|> \epsilon(1-\epsilon)^2\frac{R}{4}} |e^{\gamma |x|^2}\tilde{u}(t)|^2  dx \\ & \leq \|\tilde{u}(t)\|_{L^2(B_{\epsilon (1-\epsilon)^2\frac{R}{4}})}^2 + e^{-2\gamma \epsilon^2(1-\epsilon)^4 \frac{R^2}{16}}\sup_{t\in [0,1]}\|e^{\gamma |x|^2}\tilde{u}(t)\|_{L^2(\mathbb{R}^n)}^2,
\end{align}
and from Lemma \ref{Energyestimate},
\begin{equation}
    \frac{1}{N_V} \|\tilde{u}(0)\|_{L^2(\mathbb{R}^n)} \leq \|\tilde{u}(t)\|_{L^2(\mathbb{R}^n)} \leq N_V \|\tilde{u}(0)\|_{L^2(\mathbb{R}^n)}, \ \text{for all} \ t\in[0,1], \ N_V=e^{\sup_{t\in [0,1]}\|Im\tilde{V}(t)\|_{L^\infty(\mathbb{R}^n)}}.
\end{equation}
Combining the two inequalities above, 
\begin{align*}
    \frac{1}{N_V^2} \| \tilde{u}(0)\|_{L^2(\mathbb{R}^n)}^2 \leq \|\tilde{u}(t)\|_{L^2(B_{\epsilon(1-\epsilon)^2\frac{R}{4}})}^2 + e^{-2\gamma\epsilon^2 (1-\epsilon)^4 R^2/16} \sup_{t\in[0,1]}\|e^{\gamma |x|^2}u(t)\|_{L^2(\mathbb{R}^n)}^2.
\end{align*}
Then, integrating in time from $\frac{1-\epsilon}{2}$ to $\frac{1+\epsilon}{2}$
\begin{align*}
     \epsilon \frac{1}{N_V^2}  \| \tilde{u}(0)\|_{L^2(\mathbb{R}^n)} ^2  &\leq \|\tilde{u}\|_{L^2(B_{\epsilon (1-\epsilon)^2\frac{R}{4}} \times [(1-\epsilon)/2, (1+\epsilon)/2]} ^2 + \epsilon e^{-2\gamma\epsilon^2 (1-\epsilon)^4 R^2/16} \sup_{t\in[0,1]}\|e^{\gamma |x|^2}u(t)\|_{L^2(\mathbb{R}^n)}^2,
\end{align*}
so that, by using (\ref{bound on u in the ball}),
\begin{equation*}
    \|\tilde{u}(0)\|_{L^2(\mathbb{R}^n)} \leq N(\gamma,\epsilon,V)e^{-N_{\epsilon,\gamma,\mu}R^2}\sup_{t\in [0,1]} \|e^{\gamma |x|^2} \tilde{u} \|_{L^2(\mathbb{R}^n)}.
\end{equation*}
Since $\sup_{t\in [0,1]}\|e^{\gamma |x|^2}u(t)\|<\infty,$ we can let $R\to+\infty,$ to deduce that $\tilde{u}=0$. Going back with the Appell transformation concludes the proof of Theorem \ref{hyperbolic result}. 
\end{proof}

\begin{remark} Since $u$ is a $C([0,1],L^2(\mathbb{R}^n))$ solution, we can not guarantee that the function $g=\theta_M\eta_R \tilde{u}$ is regular enough. To justify the computations, we set $\tilde{u}_\rho= \tilde{u}*h_\rho$,where $h$ is a radial mollifier, and $g_\rho=\theta_M\eta_R \tilde{u}_\rho \in C^\infty_0(\mathbb{R}^n \times [0,1])$, we can carry out the proof for $g_\rho$, and then let $\rho\to 0.$
\end{remark}

As a direct consequence of this result we deduce the result for solutions of \eqref{HNLS}, Theorem \ref{HNLS result}:
\textit{Let $u_1$ and $u_2$ be two $C([0,1],H^k(\mathbb{R}^n))$ solutions to \eqref{HNLS}. If $k>n/2$, $F:\mathbb{C}\times \mathbb{C}\to \mathbb{C},$ $F\in C^k$, $F(0)=\partial_u F(0) = \partial_{\bar{u}}F(0) = 0$, and 
    $$\|e^{|x|^2/\alpha^2}(u_1(0)-u_2(0))\|_{L^2(\mathbb{R}^n)} + \|e^{|x|^2/\beta^2}(u_1(1)-u_2(1))\|_{L^2(\mathbb{R}^n)} <\infty$$
    for $\alpha\beta<2,$ then $u_1\equiv u_2$.}
\begin{proof}[Proof of Theorem \ref{HNLS result}]
The proof follows by writing $V=\frac{F(u_1,\bar{u}_1)-F(u_2,\bar{u}_2)}{u_1-u_2} $. Then, applying the Sobolev embedding theorem, since $k>n/2$, and the dominated convergence theorem, we have that $$\lim_{R\to \infty}\|V\|_{L^1[0,1],L^\infty(\mathbb{R}^n\setminus B_R)}=0,$$
and the result follows by applying Theorem \ref{hyperbolic result}. 
\end{proof}

\section*{Acknowledgments}
This work is a part of my Ph.D. thesis at University of Bergen, under the supervision of Didier Pilod. The research was partially supported by the Trond Mohn Foundation (TMF).

\bibliographystyle{acm}      
\bibliography{UC for HS TJ.bib}

\end{document}